\documentclass[reqno]{amsart}

\usepackage{setspace}
\usepackage{verbatim}

\usepackage{amsfonts,fancyhdr,ifthen,bm}
\usepackage{amscd,amssymb} 
\usepackage{times}
\usepackage{graphicx}
\usepackage{color}
\usepackage{epsfig}

\newtheorem{thm}{Theorem}[section]
\newtheorem*{thm*}{Theorem}
\newtheorem{prop}[thm]{Proposition}
\newtheorem{cor}[thm]{Corollary}
\newtheorem*{cor*}{Corollary}
\newtheorem{lem}[thm]{Lemma}
\newtheorem*{lem*}{Lemma}

\newtheorem*{oquest*}{Open Question}

\theoremstyle{remark}

\theoremstyle{remark}
\newtheorem*{rmk*}{Remark}

\theoremstyle{definition}
\newtheorem{defn}[thm]{Definition}

\theoremstyle{definition}
\newtheorem*{defn*}{Definition}

\numberwithin{equation}{section}

\newcommand{\vect}[1]{\bm{#1}}

\newcommand{\st}{\mid} 

\newcommand{\figures}{figures}

\newcommand{\C}{\mathbb{C}}
\newcommand{\Z}{\mathbb{Z}}
\newcommand{\E}{\mathbb{E}}
\newcommand{\R}{\mathbb{R}}
\newcommand{\T}{\mathbb{T}}

\newcommand{\q}{\vect{q}}
\newcommand{\p}{\vect{p}}
\newcommand{\e}{\vect{e}}
\newcommand{\g}{\vect{g}}
\newcommand{\h}{\vect{h}}

\DeclareMathOperator*{\Vol}{Vol}

\DeclareMathOperator*{\tr}{tr}

\begin{document}



\title{
Ball Packings with Periodic Constraints
}

\author[R.~Connelly]{Robert Connelly}
\author[J.~Shen]{Jeffrey D. Shen}
\author[A.~Smith]{Alexander D. Smith}
\thanks{This research was partially supported by a Research Experience for Undergraduates at Cornell University.}

\date{\today}

\begin{abstract}

We call a periodic ball packing in $d$-dimensional Euclidean space \emph{periodically (strictly) jammed} with respect to a period lattice $\Lambda$ if there are no nontrivial motions of the balls that preserve $\Lambda$ (that maintain some period with smaller or equal volume). In particular, we call a packing \emph{consistently} periodically (strictly) jammed if it is periodically (strictly) jammed on every one of its periods. After extending a well-known bar framework and stress condition to strict jamming, we prove that a packing with period $\Lambda$ is consistently strictly jammed if and only if it is strictly jammed with respect to $\Lambda$ and consistently periodically jammed. We next extend a result about rigid unit mode spectra in crystallography to characterize periodic jamming on sublattices. After that, we prove that there are finitely many strictly jammed packings of $m$ unit balls and other similar results.  An interesting example shows that the size of the first sublattice on which a packing is first periodically unjammed is not bounded. Finally, we find an example of a consistently periodically jammed packing of low density $\delta = \frac{4 \pi}{6 \sqrt{3} + 11} + \epsilon \approx 0.59$, where $\epsilon$ is an arbitrarily small positive number. Throughout the paper, the statements for the closely related notions of \emph{periodic infinitesimal rigidity} and \emph{affine infinitesimal rigidity} for tensegrity frameworks are also given.

\end{abstract}

\maketitle

\section{Introduction}

This paper lies at the intersection of two areas of recent interest. Principally, it is about jammed hard-ball packings, which can serve as a useful model for granular materials \cite{Torq10}. To call a packing jammed is to say that the arrangement of balls admit no motion. However, for infinite packings, this notion makes little sense without some sort of boundary condition, as we can always move the balls linearly away from some point, say, the origin. Thus, we focus on infinite packings that are periodic and define several periodic boundary conditions that each make jammedness a meaningful concept. In defining this, we run into the rigidity theory of periodic bar-and-joint frameworks, objects that have been used in recent work to model zeolites and perskovites \cite{Pow11a, Bor10, Malestein-Theran-2011}.

There is an important connection between the concepts of bar-and-joint frameworks and sphere packings that comes naturally from the study of tensegrities. We call a packing periodic with respect to a period lattice $\Lambda$ \emph{periodically jammed} with respect to $\Lambda$ if there is no nontrivial motion of the packing that maintains periodicity with respect to $\Lambda$, i.e. the only such motions result in a congruent packing. In \cite{Torq10}, this corresponds to the term \emph{collectively jammed}, which in this paper we use to denote the equivalent notion of finite packings on a torus $\E^d / \Lambda$. Periodic jamming is closely related to the concept of \emph{periodic infinitesimal rigidity} seen with bar-and-joint frameworks in \cite{Ross11, Malestein-Theran-2011, Bor10}. In addition, if there is no nontrivial motion of the packing that maintains periodicity with respect to some continuously-changing lattice $\Lambda(t)$ with nonincreasing volume, the packing is called \emph{strictly jammed}. This notion is closely related to the concept of \emph{affine infinitesimal rigidity} \cite{Pow11a, Borcea-Streinu-2011}, and is by definition stronger than that of collective jamming.

The connection between periodic jamming for packings and periodic infinitesimal rigidity for bar-and-joint frameworks is commonly made by reducing the first condition to a tensegrity framework with points at the centers of the balls and edges, called struts, between tangent balls that are constrained to not decrease in length \cite{CD12, Con08, Connelly-PackingI}. From a classic result in \cite{Roth81}, it can be shown that periodic infinitesimal rigidity on tensegrities reduces to periodic infinitesimal rigidity of the corresponding bar framework and the existence of a negative equilibrium stress on the tensgrity. In \cite{Donev1, Con08}, strict jamming again reduces to an infinitesimal condition on the tensegrity as well as the lattice. Our first main theorem proves a variant of the bar and stress condition for strict jamming and affine infinitesimal rigidity.

In this discussion, one should note that a packing or framework can be periodic with respect to a number of lattices, not just some arbritrary lattice $\Lambda$. In fact, there are a number of examples of periodic packings and frameworks that are jammed and periodically infinitesimally rigid with respect to some lattices but not others \cite{Con08}. Thus, we also look at periodic boundary conditions that do not reference a specific lattice, calling a packing \emph{consistently} periodically or strictly jammed if it is periodically or strictly jammed on all of its period lattices. We use similar terminology with regards to frameworks, although other terms have been given, such as \emph{periodically infinitesimally ultrarigid} for our term \emph{consistently periodically infinitesimally rigid} \cite{Borcea12}. We prove that a periodic framework is \emph{consistently affinely infinitesimally rigid} if and only if it is affinely infinitesimally rigid and consistently periodically infinitesimally rigid. The bar and stress condition for strict jamming gives the packing analogue: consistent strict jamming is equivalent to an arbitrary strict jamming and consistent periodic jamming.

Thus, we now consider consistent periodic jamming and the closely related problem of considering sublattices of a periodically jammed packing. The RUM spectrum used to find ``flexible modes'' of crystals modeled as ball-and-joint frameworks \cite{Weg07, Dove07, Owen11,Pow11} gives a way to characterize periodic infinitesimal ridigity on certain sublattices. We formalize this result and apply it to characterize all periodic jamming on all sublattices of a packing.

We then give a number of finiteness results on taking sublattices of a periodic packing. Most prominently, there are only finitely many strictly jammed packings with the number and radii of the translationally distinct balls fixed. We also give an interesting twenty-disk packing that can be made first unjammed on unboundedly large sublattices. Furthermore, a realization of the packing's bar framework is consistently infinitesimally rigid, but not \emph{phase periodically infinitesimally rigid}, a notion borrowed from \cite{Pow11}. The calculations require an idea borrowed from \emph{parallel drawing} in \cite{Crapo-Whiteley}. We consider infinitesimal movement of the tensegrity as motions of the edges instead of motions of the points, and so we call it an \emph{infinitesimal edge flex}. This notion is used to find consistently periodically jammed packing with the low density $\delta = \frac{4 \pi}{6 \sqrt{3} + 11} + \epsilon \approx 0.58742$, where $\epsilon$ is an arbitrarily small positive number.

In Section~2~and~3, we develop formal definitions regarding packings, and give some basic results from rigidity theory. In Section~4~and~5, we prove the bar and stress condition for strictly jammed packings, give a proof of the theorem concerning consistent strict jamming is given, and add some relevant discussion. In Section~6~and~7, we characterize periodic jamming on sublattice of a packing. Section~8~and~9 define an edge flex and give the low-density and twenty-disk packing examples. Appendix A gives the full calculation for the twenty-disk packing.

\section{Definitions and Notation}
\label{sect:def}
We consider \emph{packings} of balls in the Euclidean space $\E^d$ with disjoint interiors. For any basis $B = \{\g_1, \ldots, \g_d\}$ of $\E^d$, let $\Lambda(B)=\{\lambda_1 \g_1+\cdots+ \lambda_d \g_d \mid \lambda_i \in \Z \}$ be the lattice composed of all integral linear combinations of vectors in $B$. A packing $P$ is said to be \emph{periodic} with respect to a lattice $\Lambda$ if translations by elements in $\Lambda$ lead to an identical packing; $\Lambda$ is then called a \emph{period}. $\Lambda$ acts as an automorphism group of the packing, and we require that the packing has finitely many orbits of balls under $\Lambda$. There can be infinitely many orbits if there are vanishingly small balls, but in this paper we explicitly do not call such ill-behaved arrangements periodic packings. A periodic packing is called \emph{periodically jammed} with respect to $\Lambda$ if the only continuous motions of the balls maintaining $\Lambda$ as a period result in a congruent packing.

We can consider packings periodic with respect to a lattice $\Lambda$ as a finite packing on the torus $\T^d(\Lambda) = \E^d / \Lambda$. A packing on the torus is called \emph{collectively jammed} if the only continuous motions of the balls result in a congruent packing. Since the only continuous isometries of the torus $\T^d(\Lambda)$ are translations, a packing on $\T^d(\Lambda)$ is collectively jammed if and only if the only continuous motions are translations of the entire packing. Thus, a periodic packing in $\E^d$ is periodically jammed with respect to $\Lambda$ if and only if the only continuous motions maintaining the same period are translations of the entire packing. Note that a packing in $\E^d$ is periodically jammed with respect to $\Lambda$ if and only if the corresponding packing on $\T^d(\Lambda)$ is collectively jammed.

We say that a packing in $\E^d$ is \emph{consistently} periodically jammed if it is periodically jammed with respect to every period $\Lambda$ that it has. Analogously, we can say that a packing on the torus $\T^d(\Lambda)$ is \emph{consistently} collectively jammed if it is collectively jammed on every finite cover of the torus, $\T^d(\Lambda')$, $\Lambda'$ a sublattice of $\Lambda$. With the proposition below, we see that consistent collective jamming and consistent periodic jamming are compatible.
\begin{prop}\label{prop:consIsCons}
A packing in $\mathbb{E}^d$ that is periodic with respect to $\Lambda$ is consistently periodically jammed if and only if it is periodically jammed with respect to every sublattice of $\Lambda$. 
\end{prop}

If the lattice is generated by the basis $B = \{\g_1, \ldots, \g_d\}$, then the fundamental region of the period has volume
\begin{equation}
\Vol(\T^d(\Lambda(B))) = \det(\g_1, \g_2, \ldots, \g_d).
\end{equation}
A periodic packing is called \emph{strictly jammed} with respect to period $\Lambda$ if the only continuous motions of the balls, allowing the period to change continuously without increasing its volume, result in packings congruent to the original. Similarly, define a periodic packing to be \emph{consistently strictly jammed} if it is strictly jammed with respect to every period. We note that Proposition \ref{prop:consIsCons} can be generalized to this; a packing periodic with respect to $\Lambda'$ is consistently strictly jammed if it is strictly jammed with respect to every sublattice of $\Lambda'$.

The \emph{density} of a periodic packing is the total volume of the balls in a period divided by the volume of the period. 

\section{Results from Rigidity Theory}

We define an \emph{abstract tensegrity} to be a simple graph $G = (V; B,C,S)$ with countable vertices $V$ each of finite degree and edges $E$ being the disjoint union of subsets $B$, $C$, and $S$, which are referred to as the set of \emph{bars}, \emph{cables}, and \emph{struts}, respectively. As in Definition~2.1 from \cite{Borcea-Streinu-2011}, we may define an abstract tensegrity $G$ to be $d$-periodic with respect to $\Gamma$ if $\Gamma$ is a free abelian group of automorphisms of $G$ with rank $d$ having no fixed points and a finite number of vertex orbits. Note here that elements in $\Gamma$ are also required to preserve membership of edges in $B$, $C$, and $S$.

Then, let a \emph{tensegrity} $(G, \p)$ be a realization of an abstract tensegrity $G = (V; B,C,S)$ in $\E^d$ formed by assigning each vertex $v_i$ in $V$ to the point $\p_i$ in the countable sequence $\p = (\p_1,\p_2,\ldots)$. It is useful to consider the edges in $E = B \cup C \cup S$ as directed edges $\e_k = \p_j - \p_i$ between $i$ and $j$. We can then constrain each \emph{cable} to not increase in length, \emph{bar} to stay the same in length, and \emph{strut} to not decrease in length. A tensegrity is \emph{rigid} if the only continuous motions of the points obeying the constraints result in a congruent configuration.

Now, a tensegrity $(G, \p)$ in $\E^d$ is said to be \emph{periodic} with respect to a lattice $\Lambda$ if $G$ is $d$-periodic with respect to $\Gamma$ and $\Lambda$ is an automorphism group of the set of members of $\p$, acting on $\p$ as $\Gamma$ does on $V$. A periodic tensegrity is \emph{periodically rigid} if the only continuous motions of the points that maintain the period $\Lambda$ and obey the constraints result in a congruent configuration, i.e. translations of the entire tensegrity.

We can consider tensegrities on $\T^d(\Lambda)$ as a finite sequence of points $\p = (\p_1,\ldots,\p_n)$ and edges between points represented by vectors $\e_k$ in the sequence $\e = (\e_1,\ldots,\e_{|E|})$. Since each $\e_k$ is directed, we denote such a tensegrity as $(G, \p, \e)$ where $G$ is a finite directed multigraph with sets of bars, cables, and struts that are symmetric, i.e. containing $(u,v)$ if and only if containing $(v,u)$. A tensegrity $(G, \p, \e)$ on $\T^d(\Lambda)$ is rigid if and only if the only continuous motions of the points that obey the constraints are continuous translations of the entire tensegrity.

Now let the \emph{graph of the packing} denote the graph where vertices are placed at disk centers and edges are placed between the centers of tangent disks. We next define the \emph{corresponding strut tensegrity} as in \cite{CD12, Connelly-PackingI}; every edge in the graph of the packing becomes a strut in the corresponding tensegrity, preventing the edges from decreasing in length. Since there are only finitely many points and edges in a packing on a torus $\T^d(\Lambda)$, the packing and the strut tensegrity are related in the following way.

\begin{thm}\label{thm:col-jammed-rigid}
Given a packing $P$ in $\E^d$ periodic with respect to the lattice $\Lambda$, the following are equivalent:
\begin{enumerate}
\item[(a)] $P$ is periodically jammed with respect to $\Lambda$.
\item[(b)] The corresponding packing on $\T^d(\Lambda)$ is collectively jammed.
\item[(c)] The corresponding strut tensegrity on $\E^d$ is periodically rigid with respect to $\Lambda$.
\item[(d)] The corresponding strut tensegrity on $\T^d(\Lambda)$ is rigid.
\end{enumerate}
\end{thm}

We omit the proof of this easy theorem.

Therefore, we consider a tensegrity $(G,\p,\e)$ on a torus $\T^d(\Lambda)$. An \emph{infinitesimal flex} $\p'$ of $G$ is a sequence of vectors $(\p_1', \ldots, \p_n')$ in $\E^d$ such that for every edge $\e_k$ connecting $\p_i$ to $\p_j$ in $\T^d(\Lambda)$,
\begin{equation}
\label{eq:infFlex_cond}
\e_k \cdot (\p_j' - \p_i')
\begin{cases}
= 0, & \mbox{if } \e_k \in B \\
\leq 0, & \mbox{if } \e_k \in C \\
\geq 0, & \mbox{if } \e_k\in S
\end{cases}.
\end{equation}
On a torus, an infinitesimal flex is called \emph{trivial} if $\p_1' = \cdots = \p_n'$. A tensegrity is \emph{infinitesimally rigid} if the only flexes are trivial. We recall the following theorem, whose proof is split between \cite{Roth81} and \cite{Con08}.

\begin{thm}\label{thm:rigid-inf-rigid}
A tensegrity $(G, \p, \e)$ on $\T^d(\Lambda)$ is rigid if it is infinitesimally rigid. The converse is true if  $(G, \p, \e)$ is a strut tensegrity.
\end{thm}

A \emph{stress} on a tensegrity is a sequence of scalar weights assigned to the edges $\e_k$ such that the weight is negative on struts and positive on cables and so that the stresses on symmetric edges $\e_k$ and $-\e_k$ are the same. A stress $(\omega_1, \dots, \omega_{|E|})$ is called an \emph{equilibrium stress} if for every vertex $\p_j$,
\begin{equation}\label{eq:eq-stress}
\sum_{\e_k \in E_j} \omega_k \e_k = 0
\end{equation}
where $E_j$ is the set of edges starting at $\p_j$. The notion of a stress gives an intuition for a version of Farkas' Lemma from \cite{Roth81}.

\begin{lem}\label{lem:roth-whiteley}
Suppose $Y = \{\vect{y}_1, \ldots, \vect{y}_k\} \subset \R^d$. Then the set 
\begin{equation}
Y^+ = \{\vect{\mu} \in \mathbb{E}^d \mid \vect{\mu} \cdot \vect{y} \geq 0, \,\,\mbox{for all } \vect{y} \in Y\}
\end{equation}
is the orthogonal complement $Y^\bot$ of $Y$ if and only if there exist positive scalars $\lambda_1, \ldots, \lambda_k$ such that $\sum_{i=1}^k \lambda_i \vect{y}_i = 0$. Otherwise, $Y^+$ strictly contains $Y^\bot$.
\end{lem}

Lemma~\ref{lem:roth-whiteley} can be directly applied in the same manner as Theorem~5.2 from \cite{Roth81} to tensegrities on $\T^d(\Lambda)$ to obtain a bar and stress decomposition of infinitesimal rigidity on the quotient torus.

\begin{thm}\label{thm:bar-stress}
A tensegrity on $\T^d(\Lambda)$ is infinitesimally rigid if and only if the corresponding bar tensegrity is infinitesimally rigid and there is an equilibrium stress.
\end{thm}

Therefore, as a result of Theorems~\ref{thm:col-jammed-rigid},~\ref{thm:rigid-inf-rigid},~and~\ref{thm:bar-stress}, we have the following main theorem from rigidity theory.

\begin{thm}\label{thm:col-jammed-bar-stress}
A packing on $\T^d(\Lambda)$ is collectively jammed if and only if the corresponding bar tensegrity is infinitesimally rigid and there is an equilibrium stress on the corresponding strut tensegrity $(G, \p, \e)$.
\end{thm}

\section{Strictly Jammed Packings}
\label{sect:strict}

Let $(G,\p)$ be a tensegrity in $\E^d$ periodic with respect to the lattice $\Lambda$. In this section, we develop a number of results about the strict infinitesimal rigidity of $(G,\p)$ with respect to the lattice $\Lambda$. Now, it is convenient to denote such a periodic tensegrity as $(G, \p, \Lambda)$, where $\p = (\p_1,\ldots,\p_n)$ and $\Lambda$ generate all other points $\p_{(k,\lambda)} = \p_k + \lambda$, where $\lambda \in \Lambda$. Thus, in this section, all notions and statements of infinitesimal flexibility and rigidity are assumed to be made with respect to the given $\Lambda$. However, one should note that the results hold more generally for any period lattice of $(G,\p)$.

Consider a sequence of vectors $\p' = (\p_1', \ldots \p_n')$ and a linear transformation $A$ of $\E^d$. In an equivalent definition to that from \cite{Pow11a}, Call $(\p', A)$ an \emph{affine infinitesimal flex} of $(G, \p, \Lambda)$ if for any edge $\e_k$ connecting $\p_{(i,0)}$ to $\p_{(j,\lambda)}$,
\begin{equation}\label{eq:affif}
\e_k \cdot (\p_j' - \p_i'+A \e_k)
\begin{cases}
= 0, & \mbox{if } \e_k \in B \\
\leq 0, & \mbox{if } \e_k \in C \\
\geq 0, & \mbox{if } \e_k \in S
\end{cases},
\end{equation}
where $\e_k = \p_{(j,\lambda)}-\p_{(i,0)}$. Note here that it is convenient for us to consider the set of edges $E$ as a finite sequence of vectors $\e_k$ starting from points $\p_{(i,0)} = \p_i$, since periodicity forces all other edges to behave similarly.

This is a variant of the definition of an affine infinitesimal flex from \cite{Pow11a}, where the condition was instead on $\e_k \cdot (\p_j' - \p_i' + A\lambda)$. We see affine infinitesimal flexes of the two types correspond with each other. If $(\p', A)$ is an affine infinitesimal flex under our definition, then the flex $(\q', A)$ defined by $\q'_i = \p'_i + A \p_{(i, 0)}$ is an affine infinitesimal flex by Power's definition.

If $\Lambda$ is generated by basis vectors $\g_1,\ldots,\g_d$, then we can consider the change of basis matrix $T = (\g_1 \ldots \g_d)$. $T$ has an inverse $T^{-1}$ with rows $\h_i^T$. We see $\g_i \cdot \h_j$ is one if $i=j$ and is zero otherwise. The infinitesimal flex of the lattice generators are given by the matrix $T' = ( \g'_1 \dots \g'_d) = AT$, where we have used $\g'_i$ to represent the $i^{th}$ column of the matrix. Then we can alternatively call $(\p', \g')$ an affine infinitesimal flex if it satisfies

\begin{equation}
\e_k \cdot (\p_j' - \p_i'+\lambda_{k1}\g_1'+\cdots+\lambda_{kd}\g_d')
\begin{cases}
= 0, & \mbox{if } \e_k \in B \\
\leq 0, & \mbox{if } \e_k \in C \\
\geq 0, & \mbox{if } \e_k \in S
\end{cases},
\end{equation}
for each $\e_k$, where we have taken $\e_k = \lambda_{k1}\g_1 + \cdots + \lambda_{kd}\g_d$. We see that $\h_m \cdot \e_k = \lambda_{km}$.

The infinitesimal area condition of strict jamming, which holds that
\begin{equation}
0 \ge \frac{1}{\det(T)}\left. \frac{d}{dt} \det(T+tAT) \right|_{t=0}
\end{equation}
reduces to
\begin{equation}\label{eq:sif}
\tr(A) = \tr(T^{-1}T') = \h_1\cdot\g_1'+\cdots+\h_d\cdot\g_d' \leq 0,
\end{equation}
We call $(\p', A)$ a \emph{strict infinitesimal flex} if it is an affine infinitesimal flex satisfying \eqref{eq:sif}. (We note that this flex is called `strict' because of its relation to strict jamming and not because \eqref{eq:sif} is strictly satisfied.) An affine or strict infinitesimal flex is called \emph{trivial} if it corresponds to a rigid motion of $\mathbb{E}^d$. Then, $(G, \p, \Lambda)$ is called \emph{affinely infinitesimally rigid} if the only affine infinitesimal flexes are trivial and \emph{strictly infinitesimally rigid} if the only strict infinitesimal flexes are trivial.

\begin{prop}\label{prop:aff-strict}
A periodic bar tensegrity $(G, \p, \Lambda)$ has a nontrivial affine infinitesimal flex if and only if it has a nontrivial strict infinitesimal flex.
\end{prop}

\begin{proof}
The if direction is clear. Suppose instead that $(G, \p, \Lambda)$ has a nontrivial affine infinitesimal flex $(\p',A)$. Then, $(-\p',-A)$ is also a nontrivial affine infinitesimal flex. Since at least one of $(\p',A)$ and $(-\p',-A)$ satisfies \eqref{eq:sif}, there exists a nontrivial strict infinitesimal flex.
\end{proof}

\begin{cor}
A periodic bar tensegrity $(G, \p, \Lambda)$ is affinely infinitesimally rigid if and only if it is strictly infinitesimally rigid.
\end{cor}

We call a \emph{stress} $(\omega_1,\ldots,\omega_{|E|})$ on a periodic tensegrity $(G, \p, \Lambda)$ a \emph{strict equilibrium stress} if it is an equilibrium stress, i.e. satisfying \eqref{eq:eq-stress}, which also satisfies for $j = 1,2,\ldots, d$,
\begin{equation}\label{eq:str-eq-stress}
\sum_{\e_k \in E} \omega_k \lambda_{kj}\e_k + \h_j = 0,
\end{equation}
where $\e_k, \lambda_{kj}$ are as defined in \eqref{eq:affif} and $\h_j$ is as defined in \eqref{eq:sif}.

Equivalently, using $\lambda_{kj} = \h_j \cdot \e_k$, the condition on a strict equilibrium stress becomes
\begin{equation}\label{eq:str-eq-stressAlt}
\sum_{\e_k \in E} \omega_k \e_k \e_k^T = -I_d
\end{equation}
where $I_d$ is the $d$-dimensional identity matrix. In the exact same manner as the proof of Theorem~5.2 from \cite{Roth81}, we obtain the following bar and stress condition for strict infinitesimal rigidity.

\begin{thm}\label{thm:strict-strut-bar}
A tensegrity $(G, \p, \Lambda)$ is strictly infinitesimally rigid if and only if the corresponding bar framework has no nontrivial affine infinitesimal flexes satisfying \eqref{eq:sif} with the inequality replaced by equality, and there exists a strict equilibrium stress on $(G, \Lambda)$.
\end{thm}

\begin{proof}
For an edge $\e_k = \p_{(j,\lambda)}-\p_{(i,0)}$ going from $\p_{(i,0)}$ to $\p_{(j,\lambda)}$, consider 
\begin{equation}
\vect{y}_k = (\vect{a}_{k1},\ldots,\vect{a}_{kn},\vect{b}_{k1},\ldots,\vect{b}_{kd})
\end{equation}
as a $d(n+d)$ dimensional vector such that $\vect{a}_{ki} =  \e_k$, $\vect{a}_{kj} = -\e_k$, $\vect{b}_{km} = \lambda_{km}\e_k$, and $\vect{a}_{km} = 0$ otherwise. We also may consider the vector $\h = (\vect{0},\ldots,\vect{0},-\h_1,\ldots,-\h_d)$. For a strict infinitesimal flex $(\p',\g')$, we also consider it as a $d(n+d)$ dimensional vector $\vect{f}' = (\p_1',\ldots,\p_n',\g_1',\ldots,\g_d')$. Note that $\vect{y}_k \cdot \vect{f}' \geq 0$ if $\e_k$ is a strut, $(-\vect{y}_k) \cdot \vect{f}' \geq 0$ if $\e_k$ is a cable, and both hold if $\e_k$ is a bar. Furthermore, $\h\cdot\vect{f}' \geq 0$ from \eqref{eq:sif}. Thus, we may choose
\begin{equation}
Y = \{\h\} \cup \{-\vect{y}_k \st \e_k \in B \cup C\} \cup \{\vect{y}_k \st \e_k \in B \cup S\}
\end{equation}
so as to apply Lemma~\ref{lem:roth-whiteley}. Note that $Y^\bot$ is the set of all affine infinitesimal motions of the corresponding bar framework.

If $(G, \p, \Lambda)$ is strictly infinitesimally rigid, then $Y^+ = Y^\bot$ is the set representing all trivial motions. In addition, by the lemma, we have a linear combination of elements of $Y$ summing to zero with positive scalars, and thus a strict equilibrium stress.

Likewise, by Lemma~\ref{lem:roth-whiteley}, if there is a strict equilibrium stress, we can find a positive linear dependency among elements in $Y$. Thus, since the bar framework allows only trivial flexes, $Y^+ = Y^\bot$ only has trivial flexes, and thus $(G, \p, \Lambda)$ is strictly infinitesimally rigid.
\end{proof}

The bar framework condition may be strengthened by applying Corollary~4.2 to obtain the following corollary.

\begin{cor}\label{cor:strict-strut-bar}
A tensegrity $(G, \p, \Lambda)$ is strictly infinitesimally rigid if and only if the corresponding bar framework is strictly infinitesimally rigid and there exists a strict equilibrium stress on $(G, \p, \Lambda)$.
\end{cor}

\begin{proof}
Note that if the corresponding bar framework has a nontrivial strict infinitesimal flex, then the flex is also a nontrivial strict infinitesimal flex of the strut tensegrity. If there is no strict equilibrium stress, then by Theorem~\ref{thm:strict-strut-bar}, $(G,\p,\Lambda)$ is not strictly infinitesimally rigid.

If the tensegrity is not strictly infinitesimally rigid, then either there is a nontrivial affine infinitesimal flex satisfying \eqref{eq:sif} with equality, or there is no strict equilibrium stress.
\end{proof}

This theorem and its corollary are useful when studying strictly jammed packings by a result from \cite{Donev1, Con08}.

\begin{thm}\label{sj-sir}
A periodic packing is strictly jammed with respect to the lattice $\Lambda$ if and only if its corresponding strut tensegrity $(G,\p,\Lambda)$ is strictly infinitesimally rigid.
\end{thm}

Thus, by Theorem~\ref{sj-sir}, Corollary~\ref{cor:strict-strut-bar}, and Proposition~\ref{prop:aff-strict} we obtain the following theorem.

\begin{thm}
A periodic packing is strictly jammed with respect to the lattice $\Lambda$ if and only if the corresponding bar framework is strictly infinitesimally rigid (or affinely infinitesimally rigid) and there exists a strict equilibrium stress on the corresponding strut tensegrity.
\end{thm}
\section{Decomposing Consistent Strict Jamming}
Throughout the last section and into this section, we have repeatedly switched between viewing tensegrities as being on a torus $\T^d(\Lambda)$ and as being periodic structures on $\E^d$. This is fine, and we reiterate that a periodic tensegrity $(G, \p)$ is periodically infinitesimally rigid with respect to period $\Lambda$ if the corresponding tensegrity on $\T^d(\Lambda)$ is infinitesimally rigid. To save an adverb, we regularly skip the word `periodically'.

We call a periodic tensegrity $(G, \p)$ \emph{consistently} infinitesimally rigid if it is infinitesimally rigid with respect to all period lattices. Similarly we call it \emph{consistently} affinely infinitesimally rigid if it is affinely infinitesimally rigid with respect to all period lattices $\Lambda$ and \emph{consistently} strictly infinitesimally rigid if it is strictly infinitesimally rigid with respect to all period lattices $\Lambda$. We note that, as in Proposition \ref{prop:consIsCons}, a tensegrity with period $\Lambda'$ is consistently infinitesimally rigid if it is infinitesimally rigid with respect to every sublattice of $\Lambda'$; the same holds true for affine infinitesimal rigidity and strict infinitesimal rigidity. For bar frameworks, the term \emph{affinely infinitesimally ultraperiodically rigid}, or just \emph{ultrarigid}, has been used by Power and others to refer to the same concept. We start with a basic result about consistent affine infinitesimal rigidity that has previously been overlooked.

\begin{prop}
\label{prop:AUIRdecomp}
A bar framework with period $\Lambda$ is consistently affinely infinitesimally rigid if and only if it is affinely infinitesimally rigid with respect to $\Lambda$ and is consistently infinitesimally rigid.
\end{prop}
\begin{proof}
Consistent affine infinitesimal rigidity immediately implies affine infinitesimal rigidity with respect to any period $\Lambda$ and consistent infinitesimal rigidity, so the condition is clearly necessary. Now, suppose a framework is not consistently affinely infinitesimally rigid, having a nontrivial affine infinitesimal flex $(\p', A)$, $\p'$ periodic with respect to $\Lambda'$. We assume $\Lambda'$ is a sublattice of $\Lambda$. We take $\p'_{(i, \lambda_0)}$, $\lambda_0 \in \Lambda$ to refer to the infinitesimal flex of the point at $\p_{(i, \lambda_0)}$, noting that $\p'_{(i, 0)} = \p'_{(i, \lambda_0)}$ if $\lambda_0 \in \Lambda'$.

Assuming the framework is affinely infinitesimally rigid with respect to $\Lambda$, $\p'$  cannot be periodic with respect to the lattice $\Lambda$. Choosing some $\lambda \in \Lambda$ not in $\Lambda'$, consider the flex $\vect{q}'$ defined by
\begin{equation*}
\vect{q}'_{(i, \lambda_0)} = \p'_{(i, \lambda_0)} - \p'_{(i, \lambda +\lambda_0)}
\end{equation*}
Since $(\p', A)$ cannot be affinely periodic with respect to $\Lambda$ by the assumption, for some choice of $\lambda$, $\vect{q}'$ is nonzero. We now prove that the $\vect{q}'_{(i, \lambda_0)}$ cannot equal a constant vector $\vect{v}$ different from zero. Supposing they could, and saying $|\Lambda/\Lambda'| = N$, we would have
\begin{equation*}
\p'_{(i, \lambda_0)} = \p'_{(i, \lambda_0 + k\lambda)} + k\vect{v} = \p'_{(i, \lambda_0)} + N\vect{v}
\end{equation*}
so $N\vect{v} = 0$, so $\vect{v}$ is zero, a contradiction.

But then $(\q', 0)$ is a nontrivial affine flex periodic with respect to $\Lambda'$. To see this, suppose $\e_k$ connects $\p_{(i, \lambda_0)}$ to $\p_{(j, \lambda_1)}$. We find that
\begin{equation*}
\q'_{(i, \lambda_0)} - \q'_{(j, \lambda_1)} = (\p'_{(i, \lambda_0)} - \p'_{(i, \lambda + \lambda_0)} - \p'_{(j, \lambda_1)} + \p'_{(j, \lambda + \lambda_1)})
\end{equation*}
which equals
\begin{equation*}
-(\p'_{(j, \lambda_1)} - \p'_{(i, \lambda_0)} + A \e_k) + (\p'_{(j, \lambda + \lambda_1)}  - \p'_{(i, \lambda + \lambda_0)} + A\e_k),
\end{equation*}
so that $\e_k \cdot (\q'_{(j, \lambda_1)} - \q'_{(i, \lambda_0)}) = 0$ since $(\p', A)$ is an affine flex of the bar framework. Then $\q'$ is a nontrivial infinitesimal flex with period $\Lambda'$. Then, if a framework that is affinely infinitesimally rigid with respect to $\Lambda$ has an affine infinitesimal flex periodic with respect to $\Lambda'$, it is not infinitesimally rigid with respect to $\Lambda'$. This proves sufficiency.
\end{proof}

With a little more work, this simple result can be extended to all tensegrities.
\begin{thm}
\label{thm:SURdecomp}
A tensegrity with period $\Lambda$ will be consistently strictly infinitesimally rigid if and only if it is strictly infinitesimally rigid with respect to $\Lambda$ and is consistently infinitesimally rigid.
\end{thm}
\begin{proof}
Necessity is once again clear, so we only need to show these two properties are sufficient. Suppose a tensegrity is both strictly infinitesimally rigid with respect to $\Lambda$ and is consistently infinitesimally rigid. By Proposition \ref{prop:AUIRdecomp}, we then know that the corresponding bar framework is consistently affinely infinitesimally rigid. By Corollary \ref{cor:strict-strut-bar}, we then only need to show that the tensegrity has a strict equilibrium stress with respect to any sublattice $\Lambda'$ of $\Lambda$. Since the tensegrity is strictly infinitesimally rigid with respect to $\Lambda$, we know that there is a strict equilibrium stress periodic with respect to $\Lambda$ that satisfies \eqref{eq:str-eq-stressAlt}. This stress can also serve as a periodic equilibrium stress with respect to $\Lambda'$, and if we write $|\Lambda/\Lambda'| = N$, we find that the stress on the sublattice will satisfy 
\[ \sum_{\e_k \in E} \omega_k \e_k \e_k^T = -N I_d \]
By scaling this stress, we find an equilibrium stress periodic with respect to $\Lambda'$ that satisfies \eqref{eq:str-eq-stressAlt} and hence is strict. Then any sublattice has a strict equilibrium stress, so that the tensegrity is consistently strictly infinitesimally rigid. This proves the theorem.
\end{proof}

This last theorem is the cumulation of a long draught of rigidity theory. With the following corollary, which is a simple consequence of this last theorem and Theorem \ref{sj-sir}, we bring this theory back to an unexpected and fundamental statement about ball packings. 

\begin{cor}
\label{cor:SUJ=SJ+UJ}
A periodic ball packing in $\E^d$ is consistently strictly jammed if and only if it is both strictly jammed with respect to some period $\Lambda$ and consistently periodically jammed.
\end{cor}

\begin{figure}[thb]
\centering
\includegraphics[scale=0.28]{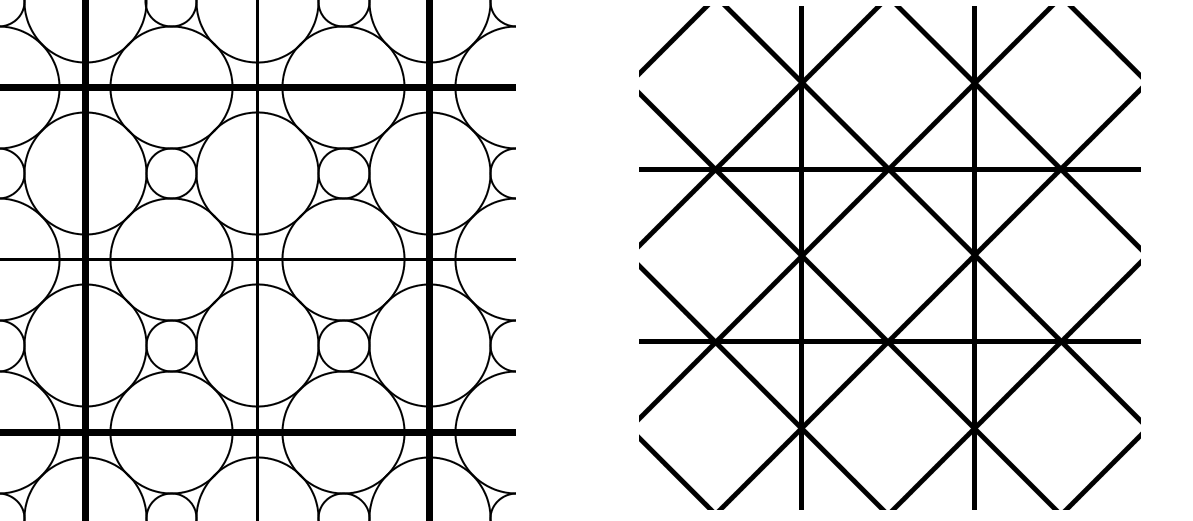}
\caption{A strictly jammed packing that is not consistently periodically jammed . The right diagram gives the graph of the packing.}
\label{fig:3disks_sJnGuJ}
\end{figure}
\begin{figure}[thb]
\centering
\includegraphics[scale=0.28]{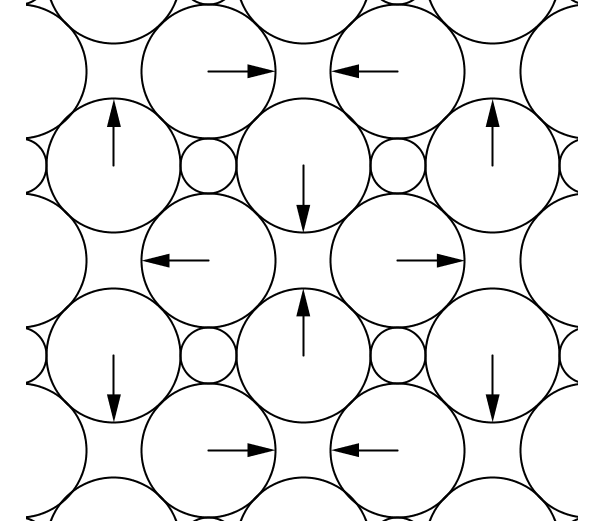}
\caption{A flex proving the example is not consistently periodically jammed ; this flex is periodic with respect to a $2\times2$ sublattice of the finest period lattice, $\Lambda$.}
\label{fig:3disks_sJnGuJ_flex}
\end{figure}

We give two observations about this result. First, we note that this result is only interesting if strict jamming and consistent periodic jamming are independent properties. It is easy to find examples of consistently periodically jammed  packings that are not strictly jammed; the packing in Figure \ref{fig:dodecagon} in Section \ref{sect:examples} is an example. Finding a strictly jammed example that is not consistently periodically jammed  is more difficult, and the existence of such a packing actually contradicts a conjecture made by Connelly in Ross's thesis, (\cite{Ross11}, p304). It is still an open question if there is a strictly jammed packing of equal radii disks that is not consistently periodically jammed. However, allowing disks of unequal sizes, there is a simple strictly jammed packing that is not consistently periodically jammed. It is depicted in Figure \ref{fig:3disks_sJnGuJ}. This packing has two larger circles and one smaller circle in its unit cell. As depicted in the figure, it is not consistently periodically jammed. It is, however, strictly jammed. It is possible to assign each edge of the contact graph the same positive stress and get a strict equilibrium stress; we leave the demonstration of this as an exercise. To prove the corresponding bar framework affinely infinitesimally rigid is also straightforward. Packings of different-radii disks have been studied as a model of alloys, and this example in particular is a variant of the $S_1$ example studied by Henley and Likos \cite{Henley-rotater} with half as many small disks. In three dimensions, there are easy examples of equal radii strictly jammed packings that are not uniformly jammed , including the octahedral network pictured in Figure \ref{fig:3spheres_sJnGuJ}. This is a similar example to the two-dimensional case, again having three balls per unit cell.

\begin{figure}[thb]
\centering
\includegraphics[scale=0.28]{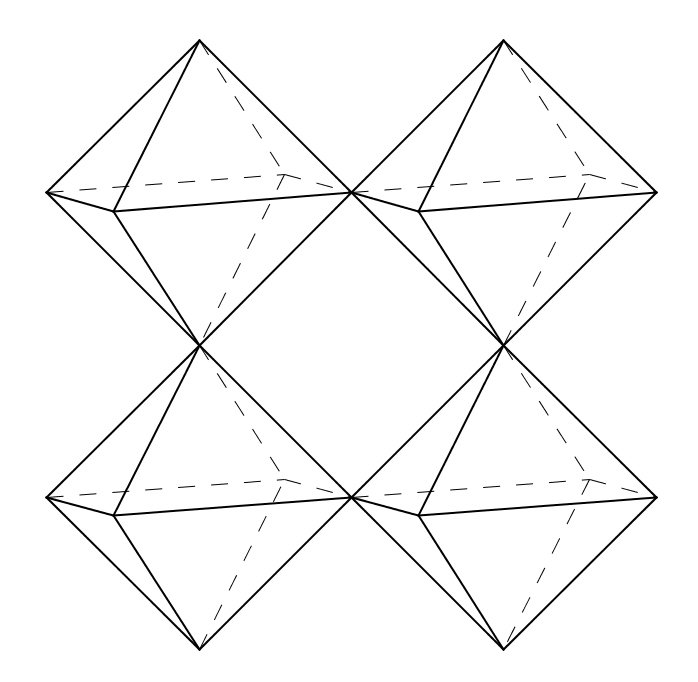}
\caption{A portion of the contact graph for the three-dimension three sphere counterexample}
\label{fig:3spheres_sJnGuJ}
\end{figure}

The second observation we make is that this result assumes spherical particles. Since there is no obvious way to model non-spherical particles with tensegrities, the introduction of non-spherical particles will make Theorem \ref{thm:SURdecomp} less useful. However, there is interest in other hard-particle systems, with Torquato surveying a large field of papers that investigate other particles in \cite{Torq10}. We leave this as an open question.

\begin{oquest*}
What conditions on particle shapes need to be assumed for consistent strict jamming to be a consequence of strict jamming and consistent periodic jamming?
\end{oquest*}
\section{Simplifying jammedness on sublattices}
\label{sect:roots}
Donev et al. developed an algorithm to determine if a packing was periodically jammed in \cite{Donev1}. To do this, they collected the $|E|$ constraints imposed by \eqref{eq:infFlex_cond}, where $|E|$ is the number of edges in the packing on the torus, into the single constraint $M \vect{p}' \ge 0$, where $M$ is known as the \emph{rigidity matrix} and has $|E|$ columns. Determining collective jammedness then reduces to a linear programming problem.

We now face the question of how we can extend this method to determine if a packing is jammed with respect to a sublattice $\Lambda'$. The na\"{i}ve approach is to treat the packing as periodic with respect to $\Lambda'$ as opposed to $\Lambda$, but this is inefficient. If $|\Lambda/\Lambda'| = N$, the rigidity matrix for the sublattice will have $N$ times as many rows and $N$ times as many columns. The linear programming problem will be corresponding more difficult.

One critical simplification comes from essentially using a discrete Fourier transform. We assume that $P$ is a periodically jammed packing that is periodic with respect to $\Lambda$ and that $\Lambda'$ is a sublattice of $\Lambda$. We identify $n$ balls in this packing as translationally distinct with respect to $\Lambda$, assigning them positions $\vect{p}_{(i,0)}$. There are $(N-1)n$ other balls that are translationally distinct with respect to the sublattice $\Lambda'$, and they are at positions $\vect{p}_{(i,\lambda)}= \vect{p}_{(i,0)} + \lambda$, where $\lambda$ is a member of the group  $\Lambda/\Lambda'$.

With this proposition, we allow complex infinitesimal flexes, which assign each vertex a infinitesimal flex in the complex vector space $\mathbb{C}^d$. Complex infinitesimal flexes are physically meaningless but simplify the theory. We note that, by taking the real part or imaginary part of a complex infinitesimal flex, we get a real infinitesimal flex.

\begin{prop}
\label{prop:FourierDecomp_general}
Let $P$ be a packing that is periodically jammed with respect to period $\Lambda$. Then $P$ is periodically unjammed with respect to period $\Lambda' \subset \Lambda$ if and only if there is a nontrivial, possibly-complex infinitesimal flex $\vect{q'}$ of the corresponding bar framework $(G, \p, \Lambda')$ and an irreducible representation $\rho : \Lambda/\Lambda' \rightarrow \C\backslash\{0\}$ so that
\begin{equation}
\label{eq:rep_cond}
\vect{q'}_{(i,\lambda_0)} = \rho(\lambda_0)\vect{q'}_{(i,0)}
\end{equation}
for all $\lambda_0 \in \Lambda/\Lambda'$.
\end{prop}

\begin{proof}

We have assumed that $P$ is periodically jammed. By Theorem \ref{thm:col-jammed-bar-stress}, this implies that there is an equilibrium stress on the corresponding strut tensegrity $(G, \p, \Lambda)$. But this stress will also clearly be a periodic equilibrium stress with respect to $\Lambda'$, so from Theorem \ref{thm:col-jammed-bar-stress} we find that the packing will be unjammed with respect to $\Lambda'$ if and only if the corresponding bar framework has an infinitesimal flex on $\Lambda'$.

If there is a flex satisfying \eqref{eq:rep_cond}, $P$ is clearly unjammed with respect to $\Lambda'$, so one direction of this proposition is trivial. Now suppose that the $P$ is unjammed with respect to $\Lambda'$, so that the bar framework has a nontrivial flex $\vect{p'}$. We can translate this flex by any element $\lambda \in \Lambda$ to find another infinitesimal flex $\vect{r}'$: defining $\vect{r}'_{(i_0, \lambda_0)} = \vect{p}'_{(i_0, \lambda_0 + \lambda)}$, and using $\vect{p}_{(i_0,\lambda_0)} = \vect{p}_{(i_0, 0)} + \lambda_0$, we find
\begin{align*}
 &(\vect{p}_{(i_0, \lambda_0)} - \vect{p}_{(i_1, \lambda_1)})
\cdot (\vect{r}'_{(i_0, \lambda_0)} - \vect{r}'_{(i_1, \lambda_1)}) \\
&= (\vect{p}_{(i_0, \lambda_0+\lambda)} - \vect{p}_{(i_1, \lambda_1+\lambda)})
\cdot (\vect{p}'_{(i_0, \lambda_0 + \lambda)} - \vect{p}'_{(i_1, \lambda_1 + \lambda)})\\
&= 0
\end{align*}

Then, if $\psi$ is any function from $\Lambda/\Lambda'$ to $\C$, the flex defined by
\begin{equation}
\label{qPrime_def}
\vect{q}'_{(i, \lambda_0)} = \sum_{\lambda  \in  {\Lambda/\Lambda'}}{\psi(\lambda)\vect{p}'_{(i, \lambda_0 - \lambda)}}
\end{equation}
will be a valid infinitesimal flex of the bar framework. Note that by having $\psi$ equal one at zero and zero everywhere else, we can recover the original $\vect{p}'$.

At this point, we can swiftly prove the proposition with representation theory. We will subsequently revisit the argument less abstractly for the case $d=2$.

First, we note that $\Lambda/\Lambda'$ is abelian, and therefore all of its irreducible representations will be one dimensional by Schur's lemma (Proposition 1.7 in \cite{Ful04}). Since this group is abelian, we also have that its conjugacy classes consist of individual group members. Then any $\psi$ will be a class function, and hence any $\psi$ will be a linear combination of characters of irreducible representations (Proposition 2.30 in \cite{Ful04}). But for a one dimensional representation $\rho$, the character $\chi(\rho)$ is equivalent to $\rho$. By choosing $\psi$ equal to one at the identity and equal to zero everywhere else, and by writing $\psi = \sum_{k} {a_k \rho_k}$ as a linear combination of irreducible representations, we find
\[ \vect{p}'_{(i, \lambda_0)} = \sum_{k}{a_k\sum_{\lambda  \in  {\Lambda/\Lambda'}}{\rho_k(\lambda)\vect{p}'_{(i, \lambda_0 - \lambda)}}}\]
Then, as $\vect{p}'$ is nonzero, we can find some irreducible representation $\rho$ that leads to a nonzero infinitesimal flex
\[ \vect{q}'_{(i, \lambda_0)} = \sum_{\lambda \in {\Lambda/\Lambda'}} \rho{(\lambda)}\vect{p}'_{(i, \lambda_0 - \lambda)}\]
But $\rho$ is a homomorphism, so for this $\vect{q}'$ we find
\[\vect{q}'_{(i, \lambda_0)} = \sum_{\lambda  \in  {\Lambda/\Lambda'}}{\rho(\lambda_0)\rho(\lambda-\lambda_0)\vect{p}'_{(i, \lambda_0 - \lambda)}} = \rho(\lambda_0)\vect{q}'_{(i,0)}\]
and we are done.
\end{proof}

We now interpret this proposition for the case that $d=2$. If the two generators for $\Lambda$ are written as $\vect{g}_1$ and $\vect{g}_2$, we can then write the two generators of $\Lambda'$ as $a\vect{g}_1+b\vect{g}_2$ and $c\vect{g}_1+d\vect{g}_2$. Then, for $\rho : \Lambda/\Lambda' \rightarrow \C\backslash\{0\}$ a homomorphism, we have $\rho(\vect{g}_1)^a\rho(\vect{g}_2)^b = 1$ and $\rho(\vect{g}_1)^c\rho(\vect{g}_2)^d = 1$. Write $\mu = \rho(\vect{g}_1)$ and $\mu' = \rho(\vect{g}_2)$. With this notation, we get a more applicable version of the above proposition.

\begin{prop}
\label{prop:FourierDecomp_specific}
A collectively jammed periodic disk packing will be unjammed with respect to $\Lambda'$ if and only if there is some choice of $\mu$ and $\mu'$
where
\[\mu^a\mu'^b = \mu^c\mu'^d = 1\]
so that there is a infinitesimal flex satisfying
\begin{eqnarray}
\label{eq:phPerFlex_cond1}
\vect{p'}_{(i,\lambda+\vect{g}_1)} = \mu\vect{p'}_{(i,\lambda)} \\
\label{eq:phPerFlex_cond2}
\vect{p'}_{(i,\lambda+\vect{g}_2)} = \mu'\vect{p'}_{(i,\lambda)}.
\end{eqnarray}
\end{prop}

\begin{figure}[thbp]
\centering
\includegraphics[scale=0.3]{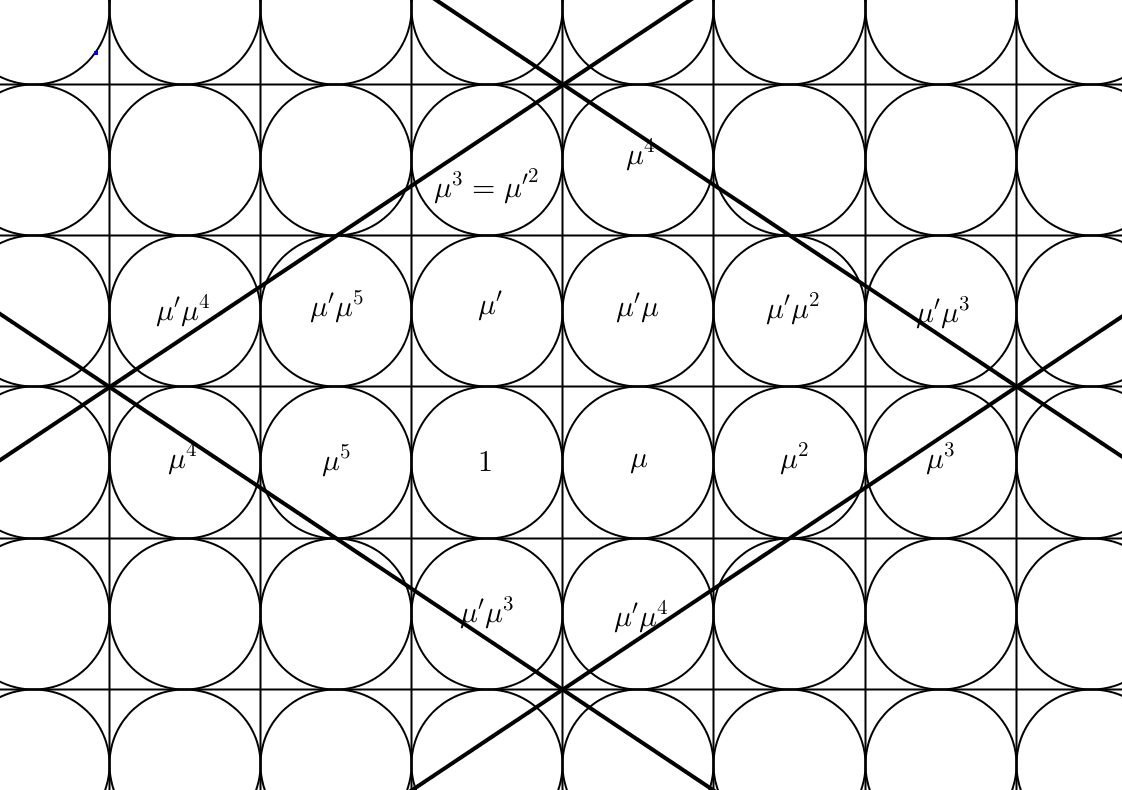}
\caption{A collectively jammed packing and a sublattice. If there is an infinitesimal flex with respect to the sublattice, Proposition \ref{prop:FourierDecomp_specific} shows there is a phase-periodic flex. The phase of each disk's flex is given in their label.}
\label{fig:1disk_rouExample}
\end{figure}

This proposition is quite similar to a result of Power in \cite{Pow11}. Even before Power's work, though, crystallographers were interested in finding the $\mu$ and $\mu'$ that would lead to ``floppy modes'' in crystals, to use Power's term. The subset of$\{(\mu, \mu') : |\mu| = 1, |\mu'|=1\}$ that lead to nontrivial infinitesimal flexes satisfying \eqref{eq:phPerFlex_cond1} and \eqref{eq:phPerFlex_cond2} is known as the rigid-unit mode (RUM) spectrum of a two dimensional crystal \cite{Pow11, Owen11}. The RUM spectrum is often used to understand zeolites, silica crystals characterized by rigid tetrahedra \cite{Pow11,  Weg07, Kap09}. We take Power's term and call an infinitesimal flex with such a $\mu$ and $\mu'$ \emph{phase periodic}.

We consider the collectively jammed packing of one disk in the square torus pictured in Figure \ref{fig:1disk_rouExample}. The lattice in this example has generators $(1,0)$ and $(0,1)$, and these are associated with roots of unity $\mu$ and $\mu'$. The generators of the sublattice in this figure are $(3, 2)$ and $(3, -2)$, so
\begin{equation}
\label{eq:12cellSublattice_rouCond}
\mu^3\mu'^2 = \mu^3\mu'^{-2} = 1
\end{equation}
Let $\vect{p}'_0$ be an infinitesimal flex on the disk marked $1$. Then there are two conditions for $(\mu, \mu')$ to be in the RUM spectrum, and they are
\[(\mu\vect{p}'_0-\vect{p}'_0)\cdot(1,0) = 0\]
and
\[(\mu'\vect{p}'_0-\vect{p}'_0)\cdot(0,1) = 0\]
This packing is collectively jammed, so the case $\mu= \mu' = 1$ is unimportant. If $\mu \ne 1$, then the x-component of $\vect{p}'_0$ is zero, and if $\mu' \ne 1$ the y-component is zero. Then, for a nontrivial infinitesimal flex to exist, exactly one of $\mu$, $\mu'$ will equal one. In particular, the packing is flexible on the sublattice in Figure \ref{fig:1disk_rouExample}, as \eqref{eq:12cellSublattice_rouCond} has the three solutions $(1, -1)$, $(e^{2\pi i /3}, 1)$, and $(e^{4\pi i /3}, 1)$. We show one infinitesimal flex in Figure \ref{fig:1disk_rouExampleFlex}.

In general, this packing will be flexible with respect to the sublattice $\Lambda'$ if one of $\mu$, $\mu'$ can be set to one without forcing the other to be one. Algebraically, this means $\gcd(a,c) \cdot \gcd(b,d) \ne 1$. Geometrically, it means that there is not both a vertical and a horizontal tour, where a tour is a path that hits every disk in the torus defined by the sublattice before returning to its starting location.
The geometric condition for this packing to be flexible was discovered by Connelly and Dickinson in \cite{CD12}, though they used other methods to arrive at the result.
\begin{figure}[htbp]
\centering
\includegraphics[scale=0.23]{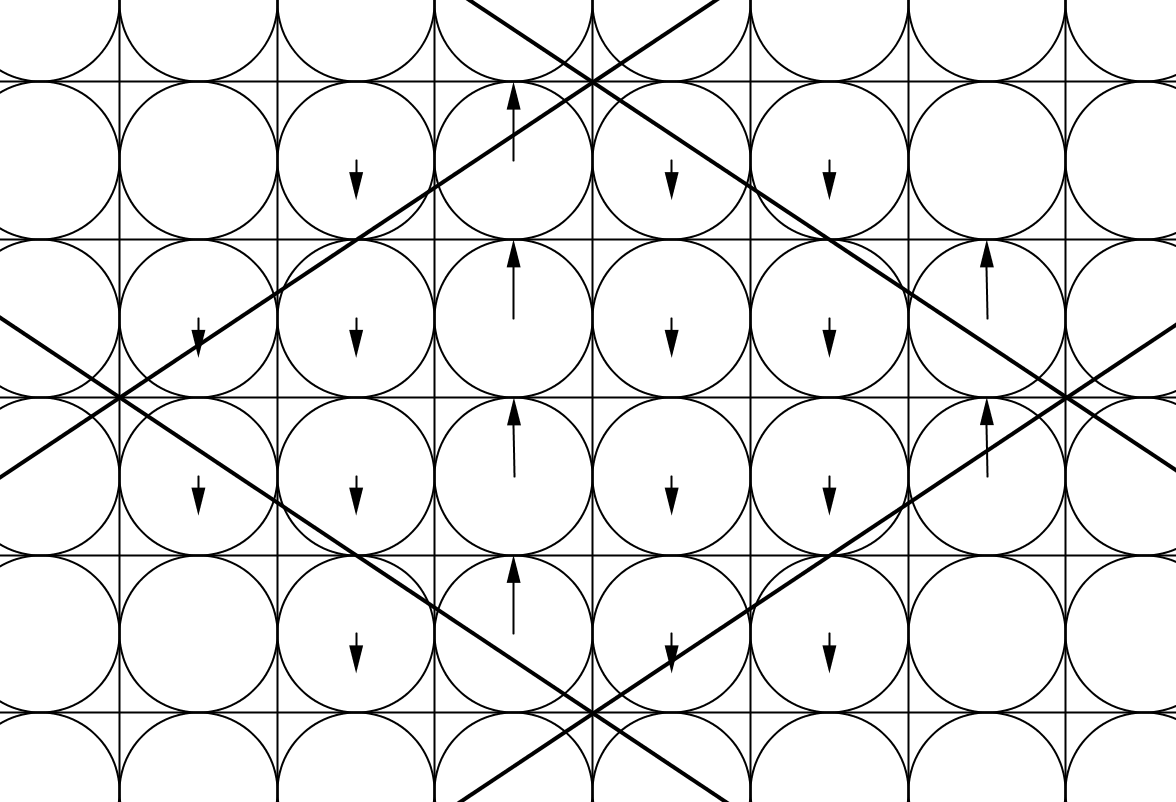}
\caption{The real part of a phase-periodic flex with phase $(\mu, \mu') =(e^{2\pi i /3}, 1)$}
\label{fig:1disk_rouExampleFlex}
\end{figure}

\subsection{A consequence of Proposition \ref{prop:FourierDecomp_general}}
\label{subsect:RUMres}

In an application of the ideas of this section, we prove an original property of the RUM spectrum.

\begin{thm}\label{thm:basicFinite}
If a collectively jammed disk packing is jammed on a $1\times k$ sublattice for some $k>1$, it will be jammed on a $1 \times k'$ lattice for any $k'$ with all prime factors sufficiently large.
\end{thm}
\begin{proof}
A phase-periodic infinitesimal flex is a solution to \eqref{eq:infFlex_cond} that obeys boundary conditions given by \eqref{eq:phPerFlex_cond1} and \eqref{eq:phPerFlex_cond2}. Then we can write the condition on infinitesimal flexes $\vect{p}'$ as
\[M(\mu, \mu') \vect{p}' = 0\]
where $M(\mu, \mu')$ is a matrix polynomial. Suppose $M$ is $n\times m$, $n\ge m$. Fix $\mu$ and $\mu'$ for the moment. If there are $m$  rows of $M$ that are linearly independent, the nullity is zero and there is no infinitesimal flex. Conversely, if no set of $m$ rows is linearly independent, there is an infinitesimal flex. Then there is an infinitesimal flex if and only if each of the $ n \choose m$ submatrices of $M$ constructed by choosing $m$ rows has a determinant of zero. Unfixing $\mu$ and $\mu'$, we can write each of these determinants as a polynomial $p_i(\mu, \mu')$, $i = 1,2,\ldots,{n\choose m}$.

Now, on a $1\times k$ sublattice, $\mu'$ must be $1$. Then either the polynomials $p_i(\mu, 1)$ are all zero for only finitely many $\mu$, or they are all zero for all $\mu$.  If they have finitely many zeroes, then there is a maximal order to the root $\mu$, and if $k'$ has no prime factor less than this order, Proposition \ref{prop:FourierDecomp_specific} implies that the packing is jammed on the $1\times k'$ sublattice. On the other hand, if the polynomials are all trivial and every $\mu$ is a root, then the packing cannot be jammed on a $1\times k$ sublattice. This proves the theorem.
\end{proof}

\section{Finiteness results for the RUM spectrum}
\label{sect:finite}

In a more-difficult extension of the ``a polynomial has finitely many zeros'' idea used in Theorem \ref{thm:basicFinite}, we can prove the following theorem, which essentially states that if a strictly jammed packing is not consistently strictly jammed, it is unjammed on some sufficiently small lattice. We associate with each collectively jammed packing a number $N_{\min}$ which equals the minimum value of $|\Lambda/\Lambda'|$ so that the packing is collectively unjammed with respect to $\Lambda'$, where $\Lambda$ is the finest period lattice. If the packing is consistently collectively jammed, we say $N_{\min} = \infty$.

\begin{thm}\label{thm:finNStrict}For any positive integer $m$ and dimension $d$ there is a positive integer $N$ so that any strictly jammed $d$-dimensional periodic packing with $m$ orbits of unit balls and $N_{min} > N$ is consistently strictly jammed. 
\end{thm}

This comes as a corollary to the theorem

\begin{thm}\label{thm:finStrict}
For any number of balls $m$ and dimension $d$, there are finitely many noncongruent periodic unit-radii packings with $m$ orbits of balls in $\mathbb{E}^d$ that are strictly jammed.
\end{thm}

In turn, this theorem comes from a more general result on bar frameworks.

\begin{thm}\label{thm:finTense}
Given an abstract bar framework $G$ with positive edge weights $d_k$ that is $d$-periodic with respect to $\Gamma$, there are finitely many noncongruent realizations of $G$ as an affinely infinitesimally rigid framework $(G, \vect{p})$ in $\E^d$ so that $\vect{p}$ maintains the periodicity of the abstract framework under $\Gamma$ and so that the length of a bar in the framework equals the weight of the corresponding edge in the abstract framework
\end{thm}

\begin{proof}
We choose $d$ generators from $\Gamma$. $G$ has a finite number of orbits under $\Gamma$; find a maximal collection of vertices $\vect{p}_1, \vect{p}_2, \dots \vect{p}_n$ distinct under $\Gamma$. The generators of $\Gamma$ correspond to generators $\vect{g}_1, \vect{g}_2, \dots \vect{g}_d$ of a lattice with respect to which the packing is jammed. The $n$ vectors $\vect{p}_i$ and $d$ vectors $\vect{g}_i$ fully determine a periodic realization.  An assignment of $\vect{p}$ and $\vect{g}$ gives a realization if and only if, for each edge $\vect{e}_k$ with weight $d_k$ connecting $\vect{p}_{i, 0}$ and $\vect{p}_{j, \lambda}$, where $\lambda = \sum_l a_l \vect{g}_l$, we have
\[ | \vect{p}_i - \vect{p}_j - \sum_{l=1}^d a_l\vect{g}_l | = d_k \]
These reduce to some finite set of quadratic equations in $d^2 + nd$ real variables, from the $d$ vectors $\vect{g}_i$ and the $n$ vectors $\vect{p}_i$. Then we can describe the set of realizations as a real algebraic set $X$ in $\mathbb{R}^{d^2 + nd}$.

From \cite{BasuPollackRoy00}, we know that a real algebraic set consists of finitely many path connected components. Suppose that two noncongruent realizations were connected by a path $\rho: [0,1] \rightarrow X$. The set of realizations congruent to $\rho(0)$ is closed, and hence so is its preimage. In particular, there exists some $a$ so $\rho(a)$ is congruent to $\rho(0)$ while $\rho(a')$ is not congruent for $a<a'< a+ \epsilon$ for some positive $\epsilon$. At this point,  we follow an approach seen in \cite{Roth81}. We use Milnor's curve selection lemma from \cite{Milnor61} to say that there is a real analytic path $\psi$ so $\psi(0) = \rho(a)$ and $\psi(x)$ is a realization not congruent to  $\rho(a)$ for $0 < x \le 1$. Taking the derivative of this path, we find a nontrivial affine infinitesimal flex. Then distinct affinely infinitesimally rigid realizations are on distinct path components. Then there are finitely many realizations.
\end{proof}

We have proved that there are only finitely many noncongruent affinely infinitesimally rigid realizations for each abstract framework. Now, to show Theorem \ref{thm:finStrict} we need to prove that there only finitely many abstract frameworks that can come from a disk packing. We need a lemma first
\begin{lem*}
Given any lattice $\Lambda$ in $\mathbb{E}^d$, there is some basis $\vect{g}_1, \dots, \vect{g}_d$ of $\Lambda$, $|\vect{g}_1| \le |\vect{g}_2| \le \dots \le |\vect{g}_d|$, so 
\[|a_1\vect{g}_1 + \dots + a_d\vect{g}_d| ^2 \ge K_d |\vect{g}_1| (a_1^2 + a_2^2 + \dots + a_n^2)\]
where $K_d$ is some positive constant that depends only on $d$.
\end{lem*}
\begin{proof}
It is a result due to Hermite that, for any dimension $d$, there is some constant $C_d$ so that any lattice $\Lambda$ in $\mathbb{E}^d$ has some basis $\vect{g}_1, \dots, \vect{g}_d$ so that the volume spanned by the vectors $\det \Lambda$ satisfies $C_d \cdot \det \Lambda \ge |\vect{g}_1| \dots |\vect{g}_d|$; in other words, the \emph{orthogonality constant} $\frac{\det \Lambda}{|\vect{g}_1| \dots |\vect{g}_d|}$ is bounded (see \cite{Cong06}). From this, the projection of $\g_i$ onto the orthogonal complement to the subspace generated by $\{ \g_1, \dots, \g_{i-1}, \g_{i+1}, \dots, \g_d \}$ has length at least $|\g_i|/C_d$. Taking $K_d = \frac{1}{dC_d^2}$ gives the lemma.
\end{proof}

\begin{thm}\label{thm:2dimAbstractCount}
For any positive $d$ and $m$, there are finitely many nonisomorphic periodic $d$-dimensional abstract frameworks that correspond to a periodic packing of $m$ unit balls in $\E^d$
\end{thm}
\begin{proof}
Consider a packing periodic with respect to $\Lambda$. We choose the $m$ disks to lie in a fundamental parallelepiped of a basis $\vect{v}_1, \dots, \vect{v}_d$ of $\Lambda$ satisfying the conditions of the lemma. Call the fundamental parallelepiped $D$. Since no unit balls can intersect, we find that $|v_1| \ge 2$. If the $m$ balls touch another ball, the touched ball must have center lying within the region 
\[ \{ \vect{v} \in \mathbb{E}^d \, \, | \, \,  \exists \vect{v}' \in D \, \, \, \mbox{such that} \, \, \, |\vect{v}-\vect{v}'| \le 2 \}\]
By the lemma, this region is a subset of 
\[\{ \vect{v} \in \mathbb{E}^d\, \,  | \, \, \vect{v} = \sum a_i\vect{v}_i,  \, \, \forall i \, \, |a_i| \le 1+ \sqrt{1/K_d} \}\]
which contains at most $A_dm$ balls, where $A_d$ is a constant dependent on $d$. Then, since the abstract framework is determined by the connections of the $m$ balls to these $A_dm$ balls through periodicity, we find that there can be no more than $2^{A_dm^2}$ abstract frameworks corresponding to ball packings, proving the theorem.
\end{proof}

Since the graph of a strictly jammed packing is affinely infinitesimally rigid as a bar framework, we see that the previous two theorems together show that there can only be finitely many strictly jammed packings of $m$ unit $d$-dimensional balls, establishing Theorem \ref{thm:finStrict} and hence Theorem \ref{thm:finNStrict}.

We can also find a variant of all four of the previous results to find a corresponding result to Theorem \ref{thm:finNStrict} for collective jamming with a fixed lattice. This looks as follows
\begin{thm} \label{thm:finN}
Given a torus $\T^d(\Lambda)$ and positive integer $m$, there is some $N$ so that any collectively jammed packing of $m$ balls on $\T^d(\Lambda)$ with $N_{min} > N$ is consistently collectively jammed.
\end{thm}

\begin{thm}
Given a torus $\T^d(\Lambda)$ and positive integer $m$, there are finitely many  noncongruent unit-radii packings of $m$ balls in $\mathbb{E}^d$ that are collectively jammed.
\end{thm}

\begin{thm}
Given some lattice $\Lambda$ and an abstract bar framework $G$ with positive edge weights $d_k$ that is $d$-periodic with respect to $\Gamma$, there are finitely many noncongruent realizations of $G$ as an infinitesimally rigid framework $(G, \vect{p})$  on $\E^d$ so that $\vect{p}$ is periodic with respect to $\Lambda$ and so that the length of a bar in the framework equals the weight of the corresponding edge in the abstract framework
\end{thm}

We omit the proof, as it uses no new ideas from the four-theorem train above.

 Finally, we note that, instead of assuming $m$ unit-radii balls, we could have assumed balls of fixed radii $r_1, \dots, r_m$; we can generalize Theorem \ref{thm:2dimAbstractCount} to this.

These finiteness results are interesting because we can push them no further. If we don't keep the extra $\Lambda$ dependence in Theorem \ref{thm:finN} and if we replace `strictly' with `collectively' in Theorem \ref{thm:finNStrict}, we get a statement that is no longer true.

\begin{thm*}
For any $N$, there is a packing of twenty disks on a torus $\T^2(\Lambda)$ that is not consistently collectively jammed but which is collectively jammed with respect to $\Lambda'$ if $ | \Lambda / \Lambda' | < N$. In other words, there is a packing that satisfies $N < N_{\min} < \infty$
\end{thm*}

We will save the proof of this result until later, as the construction needs techniques that we develop in coming sections.


\section{Edge Flexes}

In this section, we present an alternative way of viewing infinitesimal flexes that borrows from Whiteley's notion of \emph{parallel drawing} that can be found in \cite{Crapo-Whiteley}. Here, instead of considering the infinitesimal motion of vertices, we consider the infinitesimal motions of edges.

\begin{defn}\label{defn:edgeflex}
For a periodic tensegrity $(G, \p, \Lambda)$, let the sequence of vectors $\e' = (\e_1',\e_2',\ldots, \e_{|E|}')$ satisfy
\begin{equation}\label{eq:efDefA}
\e_k \cdot \e_k'
\begin{cases}
= 0, & \mbox{if } \e_k \in B \\
\leq 0, & \mbox{if } \e_k \in C \\
\geq 0, & \mbox{if } \e_k\in S
\end{cases}.
\end{equation}
Suppose further that there is some linear transformation $A$ so that, for every path $P$ of edges from $\p_{(i, 0)}$ to $\p_{(i, \lambda)}$, $\lambda$ some vector in the period lattice $\Lambda$,
\begin{equation}\label{eq:edgeflex}
\sum_{\e_k \in P} \e_k' = A \lambda
\end{equation}

Then $\e'$ is called a \emph{affine infinitesimal edge flex}. It is called a \emph{strict infinitesimal flex} if $A$ also satisfies \eqref{eq:sif}. Lastly, it is called a \emph{periodic infinitesimal edge flex} if $A= 0$. In all cases, we call the flex \emph{trivial} if $\e_k' = \vect{0}$ for all $k$.
\end{defn}

\begin{prop}\label{prop:edgeflex}
For a periodic tensegrity $(G, \p, \Lambda)$, all periodic infinitesimal edge flexes are trivial if and only if  all periodic infinitesimal flexes are trivial. That is, there exists a nontrivial periodic infinitesimal edge flex if and only if there exists a nontrivial periodic infinitesimal flex.
\end{prop}

\begin{proof}
Suppose first there is a nontrivial infinitesimal flex $\p'$. Then taking $\e_k' = \p_j' - \p_i'$ for any edge $\e_k$ from $\p_i$ to $\p_j$ gives a nontrivial infinitesimal edge flex.

Suppose instead that there is a nontrivial infinitesimal edge flex $\e'$. Let us fix $\vect{p}'_i$ to be some vector in $\E^d$. Then, for any vertex $\p_j$, the sum of $\e_k'$ along any directed path from $i$ to $j$ is some constant $\vect{c}_j'$. Letting $\p_j' = \vect{c}_j' + \p_i'$ gives a nontrivial infinitesimal flex $\p'$.
\end{proof}

The analogues for affine and strict infinitesimal edge flexes may be proved in the same manner, so we omit their proofs.

\begin{prop}\label{prop:aief}
For a periodic tensegrity $(G, \p, \Lambda)$, all affine infinitesimal edge flexes are trivial if and only if all affine infinitesimal flexes are trivial.
\end{prop}

\begin{prop}\label{prop:sief}
For a periodic tensegrity $(G, \p, \Lambda)$, all strict infinitesimal edge flexes are trivial if and only if all strict infinitesimal flexes are trivial.
\end{prop}

For planar frameworks in $\E^2$, we realize that every edge cycle can be decomposed as a sum of cycles around faces. Denoting a pair of generators for $\Lambda$ by $\g_1$ and $\g_2$, we see that every path from $\p_{(i, 0)}$ to $\p_{(i, a_1\g_1 + a_2\g_2)}$ can be decomposed as $a_1$ arbitrary paths from $\p_{(i, 0)}$ to $\p_{(i, \g_1)}$, $a_2$ arbitrary paths from  $\p_{(i, 0)}$ to $\p_{(i, \g_2)}$, and some set of cycles.  From this, we get the following proposition.

\begin{prop}\label{item:PEF_faceDecomp}
A planar periodic sequence of vectors $\e'$ is a periodic infinitesimal flex of a periodic tensegrity if and only if it satisfies \eqref{eq:efDefA} and
\begin{itemize}
\item[1)] \label{item:PEF_face}The sum of edge flexes around any face is zero.
\item[2)] \label{item:PEF_left} For some path from $\p_{(0, 0)}$ to $\p_{(0, \g_1)}$, the sum of edge flexes is zero.
\item[3)] \label{item:PEF_up} For some path from $\p_{(0, 0)}$ to $\p_{(0, \g_2)}$, the sum of edge flexes is zero.
\end{itemize}
\end{prop}

There is another simplification in the case that the planar tensegrities are bar frameworks. If every edge $\e_k$ is a bar,  we find that $\e_k'$ is normal to $\e_k$, so we may instead consider scalars $\alpha_k$ such that $\e_k' = \alpha_k R(\pi/2) \e_k$, where $R(\theta)$ is the matrix for rotation by $\theta$. An edge flex $\e'$ is then uniquely determined by the sequence of real numbers $\vect{\alpha} = (\alpha_1,\ldots,\alpha_{|E|})$. We call $\alpha_k$ the infinitesimal rotation of the edge $\e_k$ induced by the flex $\vect{\alpha}$.

The infinitesimal rotation of the edges of triangles and rhombi are then clear.

\begin{lem}\label{lem:triangle}
For any triangle determined by edges $\e_{i_1}, \e_{i_2}, \e_{i_3}$, the infinitesimal rotations $\alpha_{i_1}$, $\alpha_{i_2}$, and $\alpha_{i_3}$ are equal.
\end{lem}

\begin{lem}\label{lem:rhombus}
For a rhombus, determined by edges $\e_{i_1}, \e_{i_2}, \e_{i_3}, \e_{i_4}$, the infinitesimal rotations of parallel edges are the same, with $\alpha_{i_1} = \alpha_{i_3}$ and $\alpha_{i_2} = \alpha_{i_4}$.
\end{lem}

Both triangles and rhombi appear frequently in the bar frameworks corresponding to equal-radii packings. Thus, these two lemmas aid greatly in simplifying the calculations for jammedness of equal-radii packings, as can be seen in Section \ref{sect:examples}.
\section{Examples of Packings}\label{sect:examples}

In this section, we give a few examples of disk packings with unusual jamming properties. 

\subsection{A Low-Density Consistently Collectively Jammed Example}

Our first example of a disk packing has the contact graph seen in Figure \ref{fig:dodecagon}. 
There is a huge dodecahedral hole in this packing, leading to the low density of $\delta \approx 0.59$. However, it is still consistently collectively jammed.

\begin{figure}[htbp]
\begin{center}
\includegraphics[width=\textwidth] {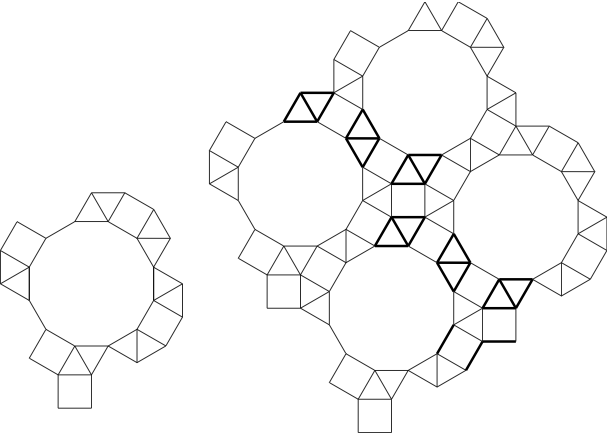}
\caption{Shown on the left is a contact graph of a single period of the packing. On the right is the contact graph of $2$ by $2$ tiling of the period of the packing. The bolded edges show equalities of infinitesimal edge rotations; we see from this that $\alpha_2 = \alpha_4 = \alpha_{12}$.}\label{fig:dodecagon}
\end{center}
\end{figure}

\begin{prop}
\label{prop:dodecagon}
The packing represented by the packing graph in Figure \ref{fig:dodecagon}
 has density $\delta = \frac{4\pi}{6 \sqrt{3} + 11} \approx 0.59$ and is consistently collectively jammed.
\end{prop}

\begin{proof}
It is simple to count $12$ equilateral triangles, $5$ squares, and $1$ regular dodecagon as the faces of the graph. There are thus $34$ edges, and $16$ vertices. If the packing radius is $\frac{1}{2}$, then the total area of the disks is $4\pi$, and the area of the torus is then $6\sqrt{3}+11$. Thus, the density is $\delta = \frac{4\pi}{6 \sqrt{3} + 11}$.

By Theorems~\ref{thm:col-jammed-rigid},~\ref{thm:rigid-inf-rigid},~and~\ref{thm:bar-stress}, it suffices to show that there exists a negative equilibrium stress and that the bar framework is infinitesimally rigid on all finite covers. Consider the dual graph with vertices at the center of each face and edges joining the centers of two adjacent faces. Assign each edge $\e_k$ a stress $\omega_k$ of the negative magnitude of the corresponding edge on its dual. This gives an equilibrium stress that works for any finite cover.

We now consider a single dodecagon in this framework. We label the rightmost edge of the dodecagon $\e_1$, the edge adjacent counterclockwise to $\e_1$ $\e_2$, and so on up to $\e_{12}$. We label the corresponding infinitesimal edge rotations $\alpha_1, \dots, \alpha_{12}$. By considering the geometry seen in Figure \ref{fig:dodecagon} and using Lemmas \ref{lem:triangle} and \ref{lem:rhombus}, we find that $\alpha_{12} = \alpha_2 = \alpha_4$, $\alpha_3 = \alpha_5 = \alpha_7$, $\alpha_6 = \alpha_8 = \alpha_{10}$, and $\alpha_9 = \alpha_{11} = \alpha_1$. Verifying these equalities is a simple and beautiful task.

Next, we know that edge flexes around the dodecagon are zero; we find that 
\[0 = R(\pi/2)(\alpha_{12}( \e_{12} + \e_2 + \e_4) + \alpha_3(\e_3 + \e_5 + \e_7)+ \alpha_6(\e_6 + \e_8 + \e_{10}) + \alpha_9(\e_9 + \e_{11} + \e_1))\]
The rotation matrix is invertible, so we can ignore it in this equation. Remembering that opposite sides of the dodecagon are parallel, this reduces to
\[ 0 = (\alpha_{12} - \alpha_6)(\e_{12} + \e_2 + \e_4) + (\alpha_3 - \alpha_9)(\e_3 + \e_5 + \e_7)\]
But $(\e_{12} + \e_2 + \e_4)$ and $(\e_3 + \e_5 + \e_7)$ are linearly independent, so we find that $\alpha_{12} = \alpha_6$ and $\alpha_3 = \alpha_9$. 
In other words, the infinitesimal rotations of the even edges are all equal and the infinitesimal rotations of the odd edges are all equal. Again by considering the geometry seen in Figure \ref{fig:dodecagon} and using Lemmas \ref{lem:triangle} and \ref{lem:rhombus}, we find that the even infinitesimal rotation for any dodecagon equals the even infinitesimal rotation for any other dodecagon and that the odd infinitesimal rotation for any dodecagon equals the odd infinitesimal rotation for every other dodecagon.

Next, no matter what period $\Lambda$ we choose, $\Lambda$ will contain some integer multiple $k$ of the vector $(\sqrt{3} + 1, -\sqrt{3} -2)$, a basis vector for the finest period lattice. The sum of the edge flexes along this path will be $kR(\pi/2)(\alpha_{12} (1, - \sqrt{3}) + \alpha_3( \sqrt{3}, -2))$. For this to equal zero, we find $\alpha_{12} = \alpha_3 = 0$. Then the only periodic infinitesimal edge flex of the bar framework is trivial, and we are done.

\end{proof}

This packing is clearly not strictly jammed, and by perturbing the lattice we find slightly denser packings on a continuum of lattices. By a more general version of the argument used in Proposition \ref{prop:dodecagon}, we find that these packings will also be consistently collectively  jammed. Now, by scaling the packing and taking an appropriate sublattice, we find a packing with density $\delta = \frac{4\pi}{6 \sqrt{3} + 11}$ on a lattice arbitrarily close to any preselected lattice. By slightly perturbing this packing, we find that, on any lattice, there is a consistently collectively jammed packing with density $\delta < \frac{4\pi}{6 \sqrt{3} + 11} + \epsilon$ for any $\epsilon > 0$. This is, in particular, an unexpectedly low density for the \emph{triangular lattice}, the lattice generated by $(1,0)$ and $(-1/2, \sqrt{3}/2)$.

\subsection{A Packing First Unjammed on an Arbitrarily Large Period Lattice}
\label{subsec:20disks}
We now have the machinery to deal with the example mentioned at the end of Section \ref{sect:finite}

\begin{thm}
\label{thm:20disks}
For any $N$, there is a packing of twenty disks on a torus $\T^d(\Lambda)$ that is not consistently collectively jammed but which is collectively jammed with respect to $\Lambda'$ if $ | \Lambda / \Lambda' | < N$. In other words, there is a packing that satisfies $N < N_{\min} < \infty$
\end{thm}

\begin{figure}[thbp]
\centering
\includegraphics[scale=0.3]{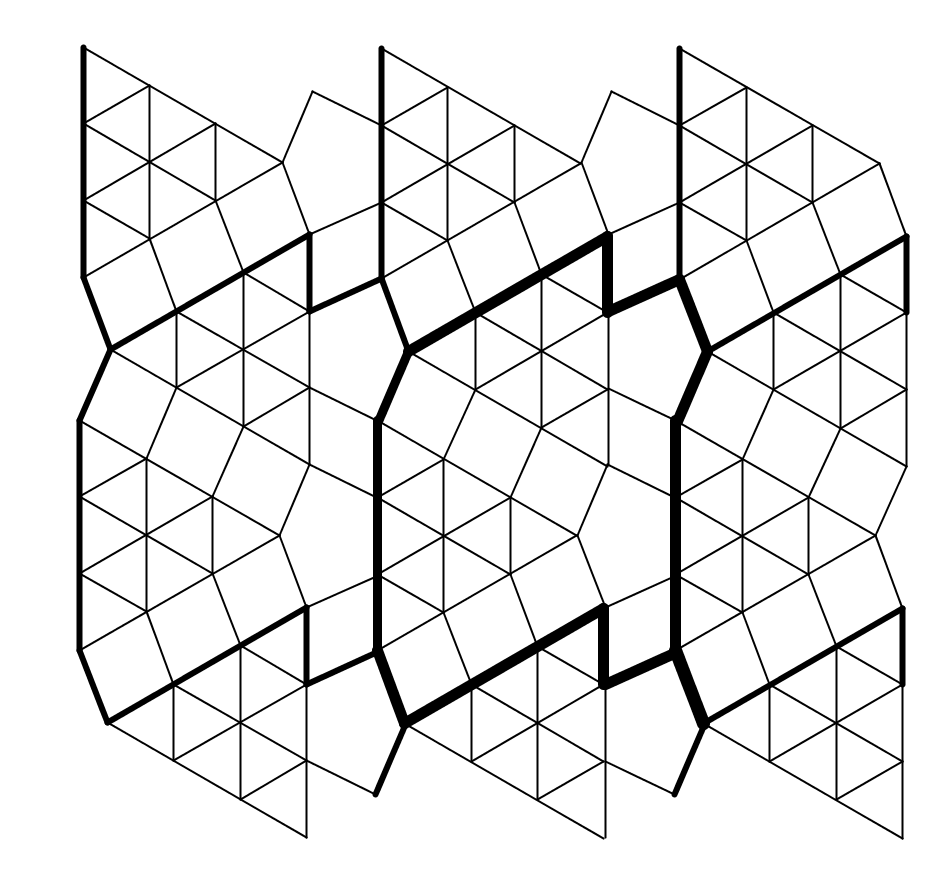}
\caption{Graph of the twenty-disk example. Some edges are bolded to emphasize the periodicity of the example.}
\label{fig:20disks}
\end{figure}

Our example is seen in Figure \ref{fig:20disks}. This packing is not strictly jammed, and by adjusting the lattice we can change the shape of the two pentagons in the lattice. The shape of the pentagon determines a crucial \emph{shape constant} $x$ that describes the flexibility properties of the packing, as we find that the packing will be flexible with respect to the phase $(\mu, \mu')$ if and only if
\begin{equation}
\label{eq:20disk_rou_flexes_new}
\frac{{\rm Re}(\mu')-1}{{\rm Re}(\mu)-1} =  x
\end{equation}

Through some technical work that we leave for the Appendix, we find that there is some interval $I\subset \R^{\ge0}$ so that for any $x\in I$ there is a collectively jammed packing with shape constant $x$. The values of $x$ arising from \eqref{eq:20disk_rou_flexes_new} are dense in $\R^{\ge0}$, and this clearly remains true if we ignore the finite number of cases where $\mu$ and $\mu'$ both have order less than $N$. Choosing one of these $x$ from the interval $I$, we find a packing with twenty disks that satisfies $N < N_{\min} < \infty$.

This packing is also interesting because its corresponding bar framework is never phase-periodically infinitesimally rigid but can be consistently infinitesimally rigid by choosing an $x$ that is never a solution to \eqref{eq:20disk_rou_flexes_new} if $\mu$ and $\mu'$ both have finite orders. 

\newpage


\bibliography{references}{}

\begin{thebibliography}{10}

\bibitem{BasuPollackRoy00}
S.~Basu, R.~Pollack, and M.-F. Roy.
\newblock {\em Algorithms in Real Algebraic Geometry}.
\newblock Springer, 2000.

\bibitem{Borcea12}
Ciprian Borcea.
\newblock Periodic rigidity.
\newblock Lecture from July 19, 2012, 09:30-10:35, July 2012.

\bibitem{Bor10}
Ciprian~S. Borcea and Ileana Streinu.
\newblock Periodic frameworks and flexibility.
\newblock {\em Proceedings of the Royal Society A}, 466:2633 -- 2649, September
  2010.

\bibitem{Borcea-Streinu-2011}
Ciprian~S. Borcea and Ileana Streinu.
\newblock Minimally rigid periodic graphs.
\newblock {\em Bull. Lond. Math. Soc.}, 43(6):1093--1103, 2011.

\bibitem{Connelly-PackingI}
Robert Connelly.
\newblock Juxtapositions rigides de cercles et de sph\`eres. {I}.
  {J}uxtapositions finies.
\newblock {\em Structural Topology}, (14):43--60, 1988.
\newblock Dual French-English text.

\bibitem{Con08}
Robert Connelly.
\newblock Rigidity of packings.
\newblock {\em European Journal of Combinatorics}, 29:1862--1871, 2008.

\bibitem{CD12}
Robert Connelly and William Dickinson.
\newblock Periodic planar disk packings.
\newblock {\em arXiv}, February 2012.

\bibitem{Crapo-Whiteley}
Henry Crapo and Walter Whiteley.
\newblock Spaces of stresses, projections and parallel drawings for spherical
  polyhedra.
\newblock {\em Beitr\"age Algebra Geom.}, 35(2):259--281, 1994.

\bibitem{Donev1}
Aleksandar Donev, Salvatore Torquato, Frank~H. Stillinger, and Robert Connelly.
\newblock A linear programming algorithm to test for jamming in hard-sphere
  packings.
\newblock {\em J. Comput. Phys.}, 197(1):139--166, 2004.

\bibitem{Dove07}
Martin~T. Dove, Alexandra K~A Pryde, Volker Heine, and Kenton~D Hammonds.
\newblock Exotic distributions of rigid unit modes in the reciprocal spaces of
  framework aluminosilicates.
\newblock {\em J. Phys.: Condens. Matter}, 19, 2007.

\bibitem{Ful04}
William Fulton and Joe Harris.
\newblock {\em Representation Theory: a first course}.
\newblock Graduate Texts in Mathematics. Springer, New York, 2004.

\bibitem{Kap09}
V.~Kapko, M.M.J. Treacy, M.~F. Thorpe, and S.~D. Guest.
\newblock On the collapse of locally isostatic networks.
\newblock {\em Proceedings of the Royal Society A}, 465:3517--3530, August
  2009.

\bibitem{Henley-rotater}
CN~Likos and C.~L. Henley.
\newblock Complex alloy phases for binary hard-disk mixtures.
\newblock {\em Philos. Mag. B.}, 68:85--113, 1993.

\bibitem{Cong06}
Cong Ling.
\newblock On the proximity factors of lattice reduction-aided decoding.
\newblock IEEE, July 2006.

\bibitem{Malestein-Theran-2011}
Justin Malestein and Louis Theran.
\newblock Generic rigidity of frameworks with orientation-preserving
  crystallographic symmetry.
\newblock {\em arXiv:1108.2518v2}, pages 1--70, 2011.

\bibitem{Milnor61}
John~W. Milnor.
\newblock {\em Singular Points of Complex Hypersurfaces}.
\newblock Annals of Mathematics Studies. Princeton UP, 1961.

\bibitem{Owen11}
John~C. Owen and Stephen~C. Power.
\newblock Infinite bar-joint frameworks, crystals and operator theory.
\newblock {\em New York J. of Math.}, 17:445 -- 490, 2011.

\bibitem{Pow11a}
Stephen~C. Power.
\newblock Crystal frameworks, symmetry and affinely periodic flexes.
\newblock {\em arXiv}, 2011.

\bibitem{Pow11}
Stephen~C. Power.
\newblock Polynomials for crystal frameworks and the rigid unit mode spectrum.
\newblock {\em Proceedings of the Royal Society A}, 2011.

\bibitem{Ross11}
Elissa Ross.
\newblock {\em Geometric and Combinatorial Rigidity of Periodic Frameworks as
  Graphs on the Torus}.
\newblock PhD thesis, York University, May 2011.

\bibitem{Roth81}
B.~Roth and W.~Whiteley.
\newblock Tensegrity frameworks.
\newblock {\em Trans. of the American Math. Soc.}, 265:419--446, June 1981.

\bibitem{Torq10}
S.~Torquato and F.~H. Stilinger.
\newblock Jammed hard-particle packings: From kepler to bernal and beyond.
\newblock {\em Reviews of Modern Physics}, 82:2633--2672, July 2010.

\bibitem{Weg07}
Franz Wegner.
\newblock Rigid-unit modes in tetrahedral crystals.
\newblock {\em J. Phys.: Condens. Matter}, 19, 2007.

\end{thebibliography}
\bibliographystyle{plain}

%
\appendix

\section{An inconsistently jammed twenty disk packing}

We now give a full proof of Theorem \ref{thm:20disks}.
\begin{thm*}
For any $N$, there is a packing of twenty disks on a torus $\T^d(\Lambda)$ that is not consistently collectively jammed but which is collectively jammed with respect to $\Lambda'$ if $ | \Lambda / \Lambda' | < N$. In other words, there is a packing that satisfies $N < N_{\min} < \infty$.
\end{thm*}
\begin{proof}
We consider the packing with the contact graph seen in Figure \ref{fig:20disks}. The flexibility of the packing depends on the shape of the two pentagons in the contact graph. In particular, if this packing is collectively jammed, we find that roots of unity $(\mu, \mu')$ will give a phase-periodic flex if and only if
\begin{equation}
\label{eq:20disk_rou_flexes}
\frac{{\rm Re}(\mu')-1}{{\rm Re}(\mu)-1} =  \frac{\cot{(\delta-\alpha)}-\cot{(\gamma-\alpha)}}{\cot{(\beta-\alpha)}-\cot{(\gamma-\alpha)}}
\end{equation}
where $\delta$, $\gamma$, and $\beta$ are the angles of the pentagon as seen in Figure \ref{fig:20disks_pentagons}. 
\begin{figure}[hl]
\centering
\includegraphics[scale=0.25]{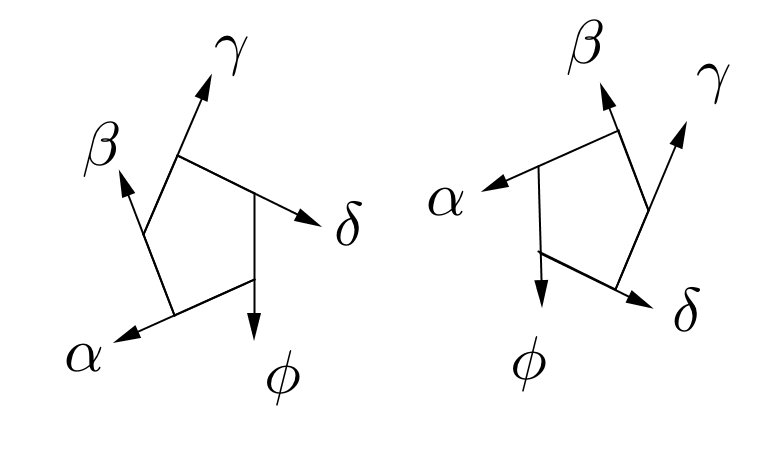}
\caption{The two pentagons. Vectors are labeled with their angle.}
\label{fig:20disks_pentagons}
\end{figure}
To prove this, we need to consider edge flexes. Suppose that there is an infinitesimal rotation hat is phase-periodic with respect to $(\mu, \mu')$; this is equivalent to a infinitesimal flex phase periodic with respect to $(\mu, \mu')$. We denote the five infinitesimal rotations of the lower pentagon in  the main cell of Figure \ref{fig:20disks_pentagons} by $a$, $b$, $c$, $d$, and $r$. Phase periodicity forces the same pentagon in the other cells to equal this times some product of powers of $\mu$ and $\mu'$, as seen in Figure \ref{fig:20disks_withEdgeFlexes}.

We start by assuming that this packing is collectively jammed. Then we only need to consider the case where one of $\mu$, $\mu'$ does not equal one. But this forces $r$ to equal zero, as $\mu'r= r$ and $\mu r = r$, as seen in Figure \ref{fig:20disks_rWalk}.
\begin{figure}[thbp]
\centering
\includegraphics[scale=0.3]{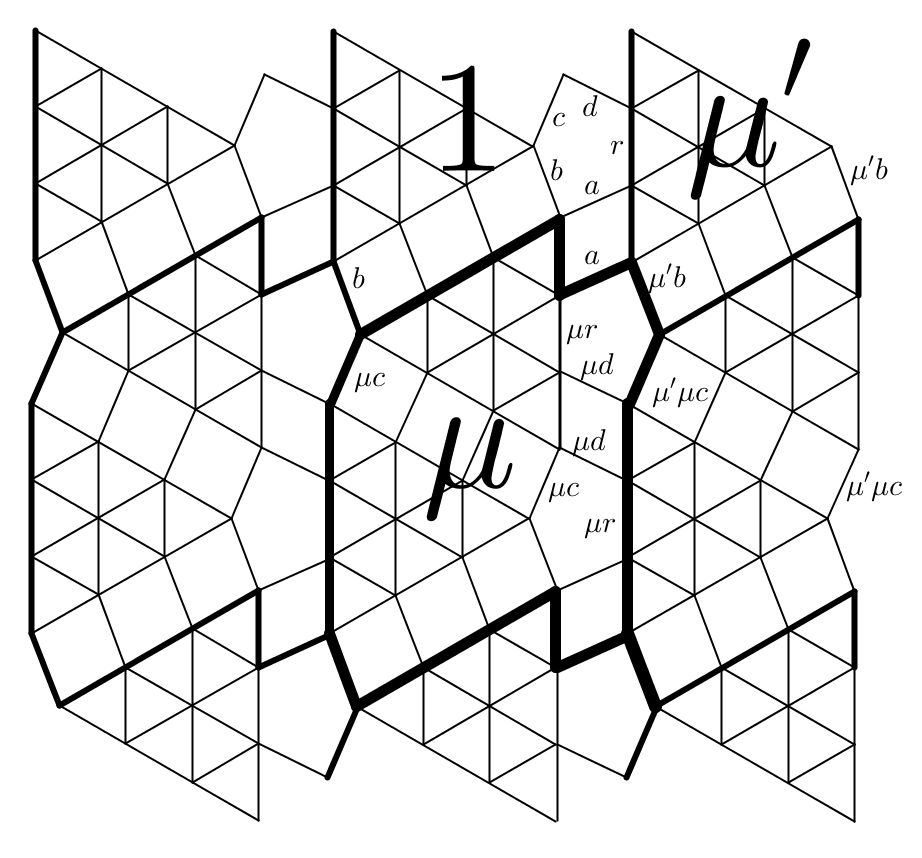}
\caption{Contact graph for twenty-disk example with edge flexes labeled.}
\label{fig:20disks_withEdgeFlexes}
\end{figure}
\begin{figure}[htbp]
\centering
\includegraphics[scale=0.33]{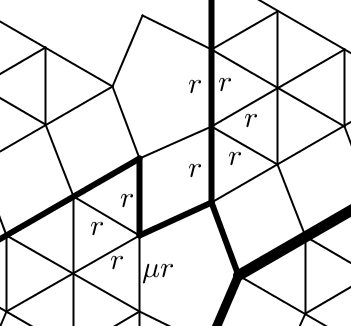}
\caption{$\mu r=r$}
\label{fig:20disks_rWalk}
\end{figure}
The flexes must sum to zero around the two pentagons, so we get
\begin{eqnarray}
\label{eq:20diskEF_cond1}
a\cos{\alpha} + b\cos{\beta} + c\cos{\gamma} + d\cos{\delta} + r\cos{\phi} = 0 \\
\label{eq:20disksEF_cond2}
a\sin{\alpha} + b\sin{\beta} + c\sin{\gamma} + d\sin{\delta} + r\sin{\phi} = 0 \\
a\cos{\alpha} + b\mu'\cos{\beta} + c\mu'\mu\cos{\gamma} + d\mu\cos{\delta} + r\cos{\phi} = 0 \\
a\sin{\alpha} + b\mu'\sin{\beta} + c\mu'\mu\sin{\gamma} + d\mu\sin{\delta} + r\sin{\phi} = 0
\end{eqnarray}
We have $r=0$, and without loss of generality we can take $\alpha=0$. Then there is an edge flex if and only if
\[ \left(
\begin{array}{cccc}
1  &  \cos{\beta} &  \cos{\gamma}  & \cos{\delta} \\
0  &  \sin{\beta}   &  \sin{\gamma}   & \sin{\delta}  \\
1 & \mu'\cos{\beta} & \mu'\mu\cos{\gamma} & \mu\cos{\delta} \\
0 & \mu'\sin{\beta} & \mu'\mu\sin{\gamma} & \mu\sin{\delta}
\end{array} \right) \]
is singular, or equivalently if it has zero determinant. The determinant of this equals
\[ \left|
\begin{array}{cccc}
1  &  \cos{\beta} &  \cos{\gamma}  & \cos{\delta} \\
0  &  \sin{\beta}   &  \sin{\gamma}   & \sin{\delta}  \\
0 & (\mu'-1)\cos{\beta} & (\mu'\mu-1)\cos{\gamma} & (\mu-1)\cos{\delta} \\
0 & (\mu'-1)\sin{\beta} & (\mu'\mu-1)\sin{\gamma} & (\mu-1)\sin{\delta}
\end{array} \right|
\] or
\[
\left|
\begin{array}{ccc}
\sin{\beta}   &  \sin{\gamma}   & \sin{\delta}  \\
(\mu'-1)\cos{\beta} & (\mu'\mu-1)\cos{\gamma} & (\mu-1)\cos{\delta} \\
(\mu'-1)\sin{\beta} & (\mu'\mu-1)\sin{\gamma} & (\mu-1)\sin{\delta}
\end{array} \right|
\]
which equals
\[\sin{\beta}\sin{\gamma}\sin{\delta}
\left|
\begin{array}{ccc}
1  &  1  & 1  \\
(\mu'-1)\cot{\beta} & (\mu'\mu-1)\cot{\gamma} & (\mu-1)\cot{\delta} \\
(\mu'-1) & (\mu'\mu-1) & (\mu-1)
\end{array} \right|
\]
If $\beta$, $\gamma$ or $\delta$ are multiples of $\pi$, we do not get a well-defined packing, so we remove the leading sine terms to this product. Taking the determinant now, we find that the condition for a flex is
\begin{eqnarray*}
(1-\mu'\mu)(1-\mu)(\cot{\gamma}-\cot{\delta}) + \\
(1-\mu)(1-\mu')(\cot{\delta}-\cot{\beta}) + \\
(1-\mu')(1-\mu'\mu)(\cot{\beta}-\cot{\gamma}) = 0
\end{eqnarray*}
or rather
\[A\mu'^2+B\mu'+C = 0\]
where $A = C = \mu(\cot{\beta}-\cot{\gamma})$ and $B$ equals
\[\mu\left(-2\cot{\beta} + (2-\frac{1}{\mu}-\mu)\cot{\delta} + (\frac{1}{\mu}+\mu)\cot{\gamma}\right)\]
Take out a factor of $\mu$, and note that $1/\mu + \mu = 2{\rm Re}(\mu)$ for $|\mu| = 1$. This makes $A$, $B$, and $C$ all real. Then clearly
\[\mu' = \frac{-B}{2A} \pm \sqrt{\left(\frac{B}{2A}\right)^2 - 1}\]
which is a well-formed expression since $A=0$ leads to degenerate packings. The product of the two roots of this quadratic is one. Notice that if $|\frac{B}{2A}| \le1$, then the two roots from this are complex conjugates, which forces the root to satisfy $\mu'\overline{\mu'} = 1$, so that $\mu'$ has absolute value one. If $|\frac{B}{2A}| > 1$, then neither root can have absolute value one. Then we get that there is a flex if and only if
\[{\rm Re}(\mu') = \frac{-B}{2A} = \frac{\cot{\beta} - {\rm Re}(\mu)\cot{\gamma} - (1- {\rm Re}(\mu))\cot{\delta}}{\cot{\beta}-\cot{\gamma}}\]
and this leads us to \eqref{eq:20disk_rou_flexes} by taking angles relative to $\alpha$.

It is clear that the values of
\[x = \frac{{\rm Re}(\mu')-1}{{\rm Re}(\mu)-1}\]
are dense in $\mathbb{R}^{\ge0}$ and remain dense if the finite number of cases where both $\mu$ and $\mu'$ have order less than $N$ are removed. Then to prove the theorem we only need to show that there is some interval $I$ in $\mathbb{R}^{\ge0}$ so that for any $x \in I$ there is some tuple $(\alpha, \beta, \gamma, \delta, \phi)$ that leads to a collectively jammed packing and which satisfies
\begin{equation}
\label{eq:20disks_keyConstant}
x = \frac{\cot{(\delta-\alpha)} - \cot{(\gamma-\alpha)}}{\cot{(\beta-\alpha)} - \cot{(\gamma-\alpha)}}
\end{equation}
To finish the theorem, we need to analyze a specific \emph{realization} of this contact graph, to borrow a term from the theory of frameworks \cite{Ross11, Bor10}. Take $AF = \arcsin\left(\frac{4}{5}\right)$, $AT = \arcsin\left(\frac{3}{10}\right)$. We consider the packing given by angles
\begin{eqnarray*}
\alpha=& 0 \\
\beta =& -AF - AT \\
\gamma =& AF - AT + \pi \\
\delta =& -2AT + \pi \\
\phi =& \frac{\pi}{2} - AT \\
\end{eqnarray*}

There are five independent, technical properties that we need to show this packing satisfies to prove the theorem. 

First, we need to show that these angles lead to a packing with the contact graph seen in Figure \ref{fig:20disks}. This is easy. The angles can describe an equilateral pentagon, for
\[e^{i\alpha} + e^{i\beta}+e^{i\gamma} + e^{i\delta} + e^{i\phi} = 0\]
and the resultant rhombi all have acute angles greater than $\pi/3$.

Second, we need to show that the angles give a positive value of $x$, but this also is trivial. By calculation, $x \approx 1.619$.

Third, we need to show that this packing has a proper negative stress. To do this, we find a reciprocal diagram in Figure \ref{fig:20disk_345stress} for a portion of the contact graph, and this reciprocal diagram will tessalate to give a reciprocal diagram for the whole packing. We generate a stress by assigning each edge in the contact graph a stress equal to negative the length of the corresponding edge in the reciprocal graph.

\begin{figure}[t]
\centering
\includegraphics[scale=0.3]{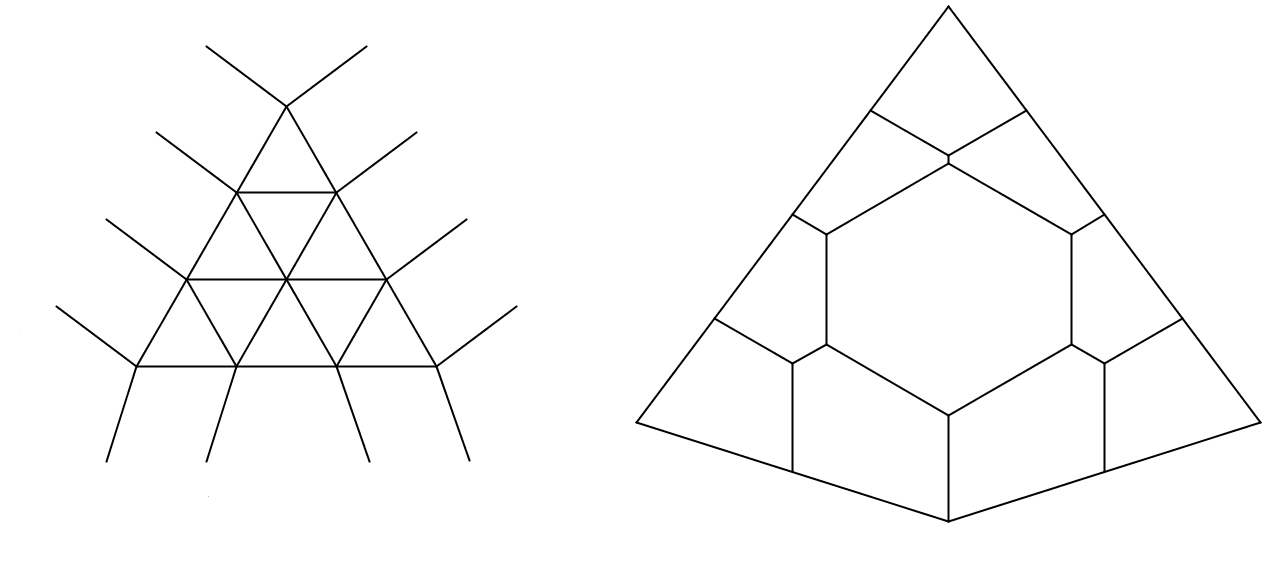}
\caption{A portion of the contact graph and a reciprocal graph.}
\label{fig:20disk_345stress}
\end{figure}

Fourth, we need to show that this packing is infinitesimally rigid as a bar framework. By Theorem \ref{thm:bar-stress}, this and the third property will prove this packing collectively jammed. We need four new conditions, and these come from considering the second and third conditions of Proposition \ref{item:PEF_faceDecomp}. Taking $\g_1$ to be the nearly vertical lattice generator, we find that the second condition gives
\begin{equation}
\label{eq:20disk_vertTour_cond1}
4r\cos{\phi} + d\cos{\delta} + a = 0
\end{equation}
and
\[4r\sin{\phi} + d\sin{\delta} = 0\]
while taking the horizontal lattice generator to be $\g_2$, we find the third condition gives
\[3r\cos{(\phi+\pi/3)}+c\cos{\gamma}+d\cos{\delta}-r\cos{\phi} = 0\]
and
\[3r\sin{(\phi+\pi/3)}+c\sin{\gamma}+d\sin{\delta} - r\sin{\phi} = 0\]
Together with \eqref{eq:20disksEF_cond2}, we get five equations on five unknowns. We solve for $a$ with \eqref{eq:20disk_vertTour_cond1} and for $b$ with \eqref{eq:20disksEF_cond2}, and this leaves us with three equations and three unknowns. The bar framework's infinitesimal rigidity then follows from the determinant of a $3\times3$ matrix being nonzero.

Fifth and finally, we need to show that $x$ is not constant in any neighborhood around this packing. To do this, we consider the differentials of the angles that arise from ``squishing'' the pentagon while maintaining its bilateral symmetry. Letting $\alpha$ be our independent variable and holding $\phi$ fixed, we find derivatives:
\[ \frac{d\phi}{d\alpha} = 0 \ \ \ \ \ \ \ \ \ \frac{d\beta}{d\alpha} = \frac{-6}{\sqrt{91}} \ \ \ \ \ \ \ \ \ \ \frac{d\gamma}{d\alpha} = \frac{6}{\sqrt{91}} \ \ \ \ \ \ \ \ \ \frac{d\delta}{d\alpha} = -1\]

Using the chain rule on \eqref{eq:20disks_keyConstant}, we find $\frac{dx}{d\alpha} \ne 0$. Then $x$ is not constant in any neighborhood of the packing, and since the packing remains collectively jammed in some neighborhood, and since $x$ changes continuously, we know that $x$ can attain any value in some interval. Since the values of $x$ that arise from the roots in $(\mu, \mu')$ having order greater than or equal to $N$ are dense in $\R^{\ge0}$, we find that there is some $x$ in this interval corresponding to a packing first flexible on a sublattice satisfying $\Lambda/\Lambda' \ge N$. This establishes the theorem.
\end{proof}

\end{document}